\documentclass[10pt, reqno]{amsart}

\usepackage{amsmath, amssymb, amsthm}
\usepackage{mathtools}
\usepackage{enumitem}
\usepackage[dvipsnames]{xcolor}
\usepackage[linktocpage]{hyperref}
\hypersetup{
	colorlinks = true,
	linkcolor=MidnightBlue,
	citecolor=BrickRed,
}
\newcommand{\NN}{\mathbb{N}}
\newcommand{\RR}{\mathbb{R}}
\newcommand{\ZZ}{\mathbb{Z}}
\newcommand{\Om}{\Omega}
\newcommand{\abs}[1]{\left\vert #1 \right\vert}
\newcommand{\norm}[1]{\left\Vert #1 \right\Vert}
\newcommand{\inner}[1]{\left\langle#1\right\rangle}
\newcommand{\D}{\mathcal{D}}
\newcommand{\N}{\mathcal{N}}
\newcommand{\CC}{\mathbb{C}}
\newcommand{\C}{\mathcal{C}}
\newcommand{\m}{\mathfrak{m}}
\newcommand{\n}{\mathfrak{n}}

\renewcommand{\d}{\mathrm{d}}
\renewcommand{\epsilon}{\varepsilon}

\numberwithin{equation}{section}

\theoremstyle{definition}
\newtheorem{defn}{Definition}[section]
\theoremstyle{remark}
\newtheorem{rem}{Remark}[section]

\theoremstyle{plain}
\newtheorem{thm}{Theorem}
\newtheorem{prop}{Proposition}[section]
\newtheorem{lem}[prop]{Lemma}
\newtheorem{cor}[prop]{Corollary}

\title[Entire Nodal Solutions of Prescribed Symmetry]{Entire Nodal Solutions with Prescribed Symmetry to Caffarelli-Kohn-Nirenberg Equations}
\author{Edward Chernysh}
\date{\today}

\begin{document}

	\begin{abstract}
		We establish the existence of sign-changing entire solutions to weighted critical $p$-Laplace equations of the Caffarelli-Kohn-Nirenberg type. In doing so, we investigate classes of symmetry and show that, for suitable  symmetry configurations, there exists a non-trivial solution which changes sign and respects the corresponding prescribed symmetry. In addition, we describe conditions under which these symmetry-types are incompatible. Especially, we demonstrate the existence of entire nodal solutions for an infinite number of distinct symmetry types.
	\end{abstract}

	\maketitle

	\section{Introduction}
	
	In this article, we investigate the existence and symmetries of sign-changing solutions to the weighted critical equations of the Caffarelli-Kohn-Nirenberg type:
	\begin{equation}\label{eq:limitProblem}
		\begin{cases}
			-\operatorname{div}\left( \abs{x}^{-ap} \abs{\nabla u}^{p-2} \nabla u\right) = \abs{x}^{-bq}\abs{u}^{q-2}u & \text{in } \RR^n,\\
			u \in \D^{1,p}_a(\RR^n,0)
		\end{cases}
	\end{equation}
	where we assume that $n \ge 4$, $1 < p < n$, and $q$ is the critical Caffarelli-Kohn-Nirenberg (henceforth abbreviated CKN) exponent, and
	\begin{equation}\label{eq:conditions}
		q := \frac{np}{n-p(1+a-b)}, \quad a < \frac{n-p}{p}. \quad \text{and} \quad a \leq b < a+1.
	\end{equation}
	Problem \eqref{eq:limitProblem} is the Euler-Lagrange equation related to the CKN inequality in $\RR^n$ (see Caffarelli-Kohn-Nirenberg \cite{caffarelli1984first}).
	The space $\D^{1,p}_a(\RR^n,0)$ is the completion of $C_c^\infty(\RR^n)$ under the homogenous weighted norm
	\(
	\norm{u}_{\D_a^{1,p}(\RR^n,0)} := \norm{\nabla u}_{L^p(\RR^n, \abs{x}^{-ap})}.
	\)
	We refer the reader to the preface of \S3, and to Chapter 2 of Chernysh \cite{mscChernysh}, for more details on $\D_a^{1,p}(\RR^n,0)$. We remark that when $a=b=0$, we simply write $\D^{1,p}(\RR^n)$. In this case, \eqref{eq:limitProblem} reduces to the critical $p$-Laplace equation 
	\begin{equation}\label{eq:unweighted}
		\begin{cases}
			-\Delta_p u = \abs{u}^{p^\ast-2}u & \text{in } \RR^n,\\
			u \in \D^{1,p}(\RR^n),
		\end{cases}
	\end{equation}
	where $\Delta_p u := \operatorname{div}\left( \abs{\nabla u}^{p-2}\nabla u\right)$ denotes the $p$-Laplacian, applied to a function $u$, and $p^\ast$ is the critical Sobolev exponent. Should we further restrict to the case $p=2$, we recover the Yamabe problem, which is itself of important interest.
	
	The existence of nodal solutions to these quasilinear elliptic critical equations have been studied extensively. Addressing literature on the existence of nodal solutions, we first and foremost mention the important work of Ding \cite{ding}, which established the existence of infinitely many sign-changing solutions of finite energy (i.e. solutions belonging to $\D^{1,2}(\RR^n)$) when $a=b = 0$ and $p=2$. Here, Ding exploited the invariance of the unweighted critical Laplace problem under M\"obius transformations. We should also take a moment to mention that, for $p=2$ and $a=b=0$, other distinct sign-changing solutions were found by Pino-Musso-Pacard-Pistoia \cite{musso} by utilizing Lyapunov-Schmidt reduction. Other solutions were found by the same authors in \cite{del2013torus}. We refer the reader to the introduction of Clapp-Rios \cite{clappRios} for details on why these approaches do not apply in the case $p \ne 2$ and more information on sign-changing solutions in this broad case. With regards to the case of the sphere, we point to works of Fern\'andez-Palmas-Petean \cite{fernandezPalmasPetean}, Fern\'andez-Petean \cite{fernandez2020low}, and Medina-Musso \cite{medinaMusso}. For further reference on the literature involving the existence of nodal solutions in the context of the Yamabe problem, we direct the reader to Premoselli-V\'etois \cite[\S2]{premo}. 
	
	With important regard to the critical $p$-Laplace equation \eqref{eq:unweighted}, we mention that Clapp-Rios \cite{clappRios} demonstrated the existence of sign-changing solutions to problem \eqref{eq:unweighted} by studying the symmetries of test functions under the action of carefully chosen subgroups of the orthogonal group $O(n)$ (see also Clapp \cite{Clapp_pure} for the case with $p=2$). An extension of this approach was recently used by Clapp-Faya-Salda\~na \cite{Clapp_2024}, for the Yamabe equation, and Clapp-Vicente-Ben\'itez \cite{clapp2026entire} to construct solutions exhibiting new symmetries for \eqref{eq:unweighted} and associated competitive systems. As for weighted equations of the Caffarelli-Kohn-Nirenberg type with subcritical terms, we observe that Szulkin-Waliullah \cite{szulkin2012sign} constructed sign-changing solutions when $0 \le a < b$.

	Our aim in this paper is twofold. First, we seek to extend the existence of sign-changing solutions to the weighted problem \eqref{eq:limitProblem} for the full range of parameters \eqref{eq:conditions}. Second, we investigate distinct types of symmetries one might impose on a sign-changing solution of \eqref{eq:limitProblem}, by generalizing the groups constructed by Clapp-Rios \cite{clappRios} and Clapp-Vicente-Ben\'itez \cite{clapp2026entire}. To formally encode these symmetries, we introduce the following notation. Given $n \ge 4$, we write $k = k(n) := \left\lfloor \frac{n}{2}\right\rfloor - 1$. We will denote by $\m = (m_1,\dots,m_k)$ a tuple of non-negative integers and, for an integer $\alpha \ge 0$, we define the conditions 
	\begin{align}
		&0 < 2\chi_\alpha  + \textstyle\sum_{j=1}^k m_j(j+1) \le n/2,\label{eq:primeCondition}\\
		&\textstyle 2\chi_\alpha + \sum_{j=1}^k m_j(j+1) \ne (n-1)/2.\label{eq:primeCondition2}
	\end{align}
	Here, $\chi_\alpha = 0$ if $\alpha = 0$ and $\chi_\alpha = 1$ otherwise. If $\alpha = 0$, we require that $\m \ne \mathbf{0}$. It should be noted that \eqref{eq:primeCondition2} is important to the case $a=b \ne 0$ and is the reason for the condition $n \ne 5$ in Theorem \ref{thm:main}.
		
	Roughly speaking, the tuple $\m$ will encode a precise symmetry configuration describing symmmetry under synchronous rotations in specified coordinate tuples, anti-symmetry  with respect to reflections and coordinate cycling, and symmetry under isometries of other coordinates. The $j$\textsuperscript{th} entry of $\m$ will define how many distinct actions occur on coordinate tuples of length $2(j+1)$. The parameter $\alpha \ge 0$ dictates the number of ``pinwheel-type''-symmetries the function will possess (see also Clapp-Vicente-Ben\'itez \cite{clapp2026entire}). Essentially, the parameter $\alpha$ is responsible for bringing an infinity of distinctive $\m$-type symmetries to solutions of \eqref{eq:limitProblem}. We refer the reader to \S2 for the precise definition of an $(\alpha,\m)$-symmetric function. However, we point out that the solutions of \eqref{eq:unweighted} obtained in Clapp-Rios \cite{clappRios} are all of $(0,\m)$-type symmetry for $\m = (j,\mathbf{0})$, with $j \ge 1$. In the case of Clapp-Vicente-Ben\'itez \cite{clapp2026entire}, no information is available on the symmetries the solutions satisfy in the last $(n-4)$-coordinates, although the $(\alpha,\m)$-symmetric solutions we produce are certain to include solutions with new symmetries.

	Our main result states that there exist solutions with $(\alpha,\m)$-symmetry, corresponding to every pair $(\alpha,\m)$ satisfying \eqref{eq:primeCondition} and, if $a=b\ne0$, condition \eqref{eq:primeCondition2}.

	\begin{thm}\label{thm:main}
		Fix a dimension $n\ge 4$, let $\alpha \ge 0$ an integer, and $\m$ be a tuple of non-negative integers such that $(\alpha,\m)$ satisfies \eqref{eq:primeCondition}.
		\begin{enumerate}
			\item  When $a < b$ or $a=b=0$, there exists a non-trivial sign-changing solution $u_{\alpha,\m}$ to \eqref{eq:limitProblem} which is $(\alpha,\m)$-symmetric.
			\item  When $a =b \ne 0$, $n \ne 5$, and $(\alpha,\m)$ satisfies condition \eqref{eq:primeCondition2}, there exists a non-trivial sign-changing solution $u_{\alpha,\m}$ to \eqref{eq:limitProblem} which is $(\alpha,\m)$-symmetric.
		\end{enumerate}
		Moreover, for any $(\beta,\n)$ satisfying \eqref{eq:primeCondition}:
		\begin{enumerate}[label=(\roman*)]
			\item if $\alpha \ne \beta$ then $u_{\alpha,\m} \ne u_{\beta,\n}$,
			\item if $\alpha = \beta$ and $\m \ne \n$, then $u_{\alpha,\m} \ne u_{\beta,\n}$ provided $\gcd(\m,\n) =1$ or $\m \lesssim \n$.
		\end{enumerate}
	\end{thm}
	
	In this theorem, we are defining for $\m \ne \n$ 
	\begingroup
	\small
	\begin{equation}
		\gcd(\m,\n) := \begin{cases} 
			\max\left\{  \gcd(j+1, \ell + 1) : m_j \ne 0,\, n_\ell \ne 0, \ell\ne j  \right\} & \text{if $\m, \n \ne 0$,}\\
			1 & \text{otherwise}.
		\end{cases}
	\end{equation}
	\endgroup
	In addition, we  declare that $\m \lesssim \n$ if there exists an index $\ell$ such that $m_j = n_j$ for all $j < \ell$, $m_\ell < n_\ell$, and $m_j = 0$ for all $\ell > j$. Notice that if $\m \ne \n$ and the coordinates of $\m$ and $\n$ are non-zero only on indices $j$ for which $(j+1)$ is prime, then automatically $\gcd(\m,\n) = 1$.
	
	To the best of the author's knowledge, Theorem \ref{thm:main} is the first result giving the existence of entire nodal solutions to the Caffarelli-Kohn-Nirenberg problem \eqref{eq:limitProblem} for the whole range of parameters \eqref{eq:conditions}.
	Aside from this important existence question, the main strength of Theorem \ref{thm:main} is in its ability to prescribe various mutually exclusive symmetry types to solutions of \eqref{eq:limitProblem}. In fact, by varying $\alpha \ge 0$ for each fixed tuple $\m$, Theorem \ref{thm:main} guarantees the existence of $(\alpha,\m)$-symmetric entire nodal solutions to \eqref{eq:limitProblem} for an infinite number of distinct symmetry configurations. We wish to emphasize that, without some condition such as (i)-(ii) in this last theorem, it is impossible to distinguish the constructed solutions by their symmetries alone. We point the reader towards Remark \ref{rem:counterExample} for an instance of a non-trivial function simultaneously satisfying a $(0,(2,\mathbf{0}))$-symmetry and a $(0,(0,0,1,\mathbf{0}))$-symmetry. Notice that in this case, both (i) and (ii)  fail to apply. As for the solutions constructed in Szulkin-Waliullah \cite{szulkin2012sign}, we mention that the symmetries of these solutions are exclusive to two coordinates only.
	
	We also observe that, to the author's knowledge, Theorem \ref{thm:main} guarantees the existence of new sign-changing solutions to the critical $p$-Laplace equation \eqref{eq:unweighted} and, especially, the Yamabe problem. What is more, for fixed $\alpha \ge 0$, if we restrict ourselves only to $\m$ such that $m_j \ne 0$ if and only if $(j+1)$ is prime, then we automatically obtain solutions with incompatible symmetries by Theorem \ref{thm:main}-(ii). In the case of symmetries acting exclusively on a prime number of complex coordinates, the possible choices $S(n)$ of $\m$ satisfying \eqref{eq:primeCondition}-\eqref{eq:primeCondition2} is asymptotically in the order of 
	\[
	S(n) \sim \exp\left( \frac{\pi\sqrt{2n}}{\sqrt{3\ln(n/2)}}\right)
	\]
	by Ramanujan \cite[Page 260]{ramanujan2015collected}.
	
	Next, we wish to acknowledge the presence of an additional requirement for our construction of a non-trivial nodal solution to \eqref{eq:limitProblem} when $a=b \ne 0$. This requirement is a direct consequence of the particularities present in the bubble tree decomposition for Palais-Smale sequences established in Chernysh \cite{Chernysh} for the Caffarelli-Kohn-Nirenberg problem \eqref{eq:limitProblem}. It is the nature of this decomposition, together with the inherent lack of translation invariance for \eqref{eq:limitProblem}, which contributes this technical requirement. We note that this detail is further elaborated upon in \S1.1, where we discuss the proof aspects with greater precision. 
	
	Bringing our attention to positive solutions of \eqref{eq:limitProblem}, much can be said regarding their classification. Firstly, when $a=b=0$, all positive solutions to problem \eqref{eq:limitProblem} have been shown to be radially symmetric and an explicit form has been determined. This was demonstrated by Caffarelli, Gidas and Spruck \cite{CaffarelliGidasSpruck} in the case $p=2$,  by Damascelli, Merch\'an, Montoro and Sciunzi \cite{damascellimerchanmontorosciunzi} for $\frac{2n}{n+2}\le p < 2$, by V\'etois \cite{vetois2016} for $1<p<2$, and finally by Sciunzi \cite{SCIUNZI201612} for $2 < p < n$. We also observe that significant work has been put forth in the direction of extending this classification to positive solutions of \eqref{eq:unweighted}, without the assumption that the solutions lie in $\D^{1,p}(\RR^n)$. Indeed, we point to the papers of Catino-Monticelli-Roncoroni \cite{catinoMonticelliRoncorini}, Ou \cite{ou2025classification}, V\'etois \cite{vetois2024note} and Sun-Wang \cite{sun2025criticalquasilinearequationsriemannian}. Recently, Ciraolo-Gatti \cite{ciraolo2025classificationresultsboundedpositive} established such a classification for positive local solutions of \eqref{eq:unweighted}, under certain boundedness or growth conditions.
	
	In the weighted setting (i.e. when $(a,b) \ne (0,0)$), the classification of positive solutions to \eqref{eq:limitProblem} has been established in certain special cases. For instance, when $b \ne a = 0$,  Ghoussoub-Yuan \cite{ghoussoub-yuan} determined an explicit form for  every positive radial solution of \eqref{eq:limitProblem}. In the case $p=2$, Dolbeault-Esteban \cite{DolbeaultEsteban} provided an explicit form (demonstrating, in particular, radial symmetry) for every positive solution to \eqref{eq:limitProblem} in the non-limiting case $a<b$, provided $a$ and $b$ satisfy certain conditions. In dimensions $n\ge 3$, Ciraolo-Corso \cite{ciraolo-corso} extended this classification  to include all $1<p<n$ when $a=b \ge 0$, as well as a wider range of possible values for $a,b$ provided $p<n/2$.  Additionally, Le-Le \cite{LeLe} later showed that, when $a=0$, all positive solutions of \eqref{eq:limitProblem} are radial and therefore classified by Ghoussoub-Yuan \cite{ghoussoub-yuan}.  With regards to classifying positive local solutions of \eqref{eq:limitProblem} (e.g. positive solutions which, a priori, need not live in $\D_a^{1,p}(\RR^n,0)$), we direct the reader to the paper of Ciraolo-Polvara \cite{ciraoloPolvara}. We refer the reader to \S1.1.1 of \cite{Chernysh} for further historical context regarding the classification of solutions to \eqref{eq:limitProblem}.\newline

	\noindent\textbf{Acknowledgements}. I would like to thank my PhD advisors, Professor J\'er\^ome V\'etois and Professor Pengfei Guan, for their invaluable support and guidance throughout the development of this work. 
	
	\subsection{Notation and Proof Discussions}
	
	Throughout this paper, $G$ denotes a closed subgroup of $O(n)$. Given a point $x \in \RR^n$, we shall write $\operatorname{Orb}_G(x)$ and $\operatorname{Stab}_G(x)$ to denote the orbit of $x$ and stabilizer of $x$, respectively, under the action of $G$. A subset $\Om \subseteq \RR^n$ is said to be $G$-\emph{invariant} provided 
	\[
	\operatorname{Orb}_G(\Om) := \bigcup_{x \in \Om} \operatorname{Orb}_G(x) \subseteq \Om.
	\]
	Of course, this is equivalent to requiring that $\operatorname{Orb}_G(\Om) = \Om$. If $\Om$ is a $G$-invariant domain, a function $u : \Om \to \RR$ will be called $G$-invariant if 
	\[
	u(gx) = u(x), \quad \forall g \in G,\, x \in \Om.
	\]
	For the remainder of this paper, whenever a subgroup $G$ of $O(n)$ is present, $\Om \subseteq \RR^n$ will be assumed to be $G$-invariant containing the origin.
	
	Given a subgroup $G$ of $O(n)$, a continuous group homomorphism $\phi : G \to \{\pm 1\}$, and a $G$-invariant domain $\Om \subseteq \RR^n$, a test function $u \in C_c^\infty(\Om)$ will be called $\phi$-equivariant if
	\[
	u(gx) = \phi(g)u(x), \quad \forall g\in G,\, x \in \Om.
	\]
	We shall denote by $C_c^\infty(\Om)^\phi$ the vector space consisting of all $\phi$-equivariant functions in $C_c^\infty(\Om)$. For the proof of Theorem \ref{thm:main}, we construct suitable subgroups $G$ of $O(n)$ along with continuous group homomorphisms $\phi : G \to \{\pm 1\}$ such that the following important properties hold.
	\begin{enumerate}
		\item[(P1)]\label{P1} There exists $\xi \in \RR^n$ such that $\operatorname{Stab}_G(\xi) \subseteq \ker{\phi}$.
		\item[(P2)]\label{P2} $\phi$ is surjective.
	\end{enumerate}
	These two properties follow those employed within Clapp-Rios \cite{clappRios}. Property (P1) is crucial to the argument as it ensures that the space of $\phi$-equivariant test functions $C_c^\infty(\Om)^\phi$ is non-trivial, and hence infinite dimensional; see Bracho-Clapp-Marzantowicz \cite[page 195]{bracho2005symmetry}. For the limit case $a=b \ne 0$ in the problem parameters, we must impose the following additional constraint on the subgroup $G$:
	\begin{enumerate}
		\item[(P3)] For each $x \in \RR^n \setminus \{0\}$, the orbit $\operatorname{Orb}_G(x)$ is infinite.
	\end{enumerate}
	Condition (P3) is a subtle, and yet important, distinction from that present in Clapp-Rios \cite{clappRios}. Indeed, within this paper, the authors allow the subgroup $G$ to satisfy either $\dim(\operatorname{Orb}_G(x)) > 0$ or $\operatorname{Orb}_G(x) = \{x\}$, for every $x \in \RR^n$. However, within the weighted case with $a=b \ne 0$, it is crucial that we ensure the orbit $\operatorname{Orb}_G(x)$ be infinite for $x \ne 0$. This is because, by the Struwe-type decomposition for weighted critical $p$-Laplace equations from \cite{Chernysh}, solutions to \eqref{eq:limitProblem} when $a=b \ne 0$ are produced precisely when concentration occurs at the origin at a sufficient rate (see \cite[Theorem 2]{Chernysh}). Consequently, in order to ensure that our process yields solutions to the weighted problem in $\RR^n$, we must enforce concentration at the origin by requiring that the origin be the only point fixed by the action of $G$. Without the strengthened condition (P3), there is the possibility that our procedure produces solutions to an unweighted problem, either in $\RR^n$ or a half-space.  In sharp contrast, however, when $a < b$, all concentration must automatically take place at the origin (see \cite[Theorem 1]{Chernysh}) which renders condition (P3) superfluous, thereby allowing us to construct more solutions.

	The remainder of this paper is delegated to the proof of Theorem \ref{thm:main}. In the next section, we construct suitable subgroups $G$ of $O(n)$ along with associated homomorphisms $ G \to \{\pm 1\}$. Our construction generalizes that of Clapp-Rios \cite{clappRios} and Clapp-Vicente-Ben\'itez \cite{clapp2026entire}, most notably by considering combinations of group actions on complex tuples whose length is any integer larger than $1$, rather than only actions on pairs of complex coordinates.   Following this, we prove that solutions may be distinguished by their symmetry-types, i.e. we establish the second part of Theorem \ref{thm:main}. To achieve this, we utilize a binary coding scheme to encode the coordinates with respect to which our equivariant solutions are rotationally invariant in order to deduce that these solutions satisfy mutually exclusive symmetry types.
	
	Afterwards, in \S3, we show that symmetric Palais-Smale sequences may be constructed at suitable energy levels which we shall exploit to demonstrate the existence of sign-changing solutions. Using these sequences, we employ a modified symmetrized analogue of the Struwe-type decomposition for weighted critical $p$-Laplace equations to extract a non-trivial solution to a symmetrized version of \eqref{eq:limitProblem}. We then show that this solution also solves the problem \eqref{eq:limitProblem} itself. 
	
	\section{Constructing Admissible Symmetry Groups}
			
	In this section, we construct the subgroups of $O(n)$ needed to produce the symmetric Palais-Smale sequences required within \S3-4 of this paper. As stated in the introduction, our groups generalize those obtained by Clapp-Rios \cite{clappRios} and Clapp-Vicente-Ben\'itez \cite{clapp2026entire} by incorporating actions involving non-trivial tuples of complex coordinates of arbitrary length.
	
	Henceforth, we fix a dimension $n \ge 4$ and put $k = k(n) := \left\lfloor n/2\right\rfloor - 1$.
	
	\subsection{Building Groups of Prescribed Symmetry}
	
	For each integer $j\ge 1$, we define an $\RR$-linear isometry of $\RR^{2(j+1)} \cong \CC^{j+1}$ given by 
	\begin{equation}
		\rho_{j} : \CC^{j+1} \to \CC^{j+1}, \quad (z_1,\dots, z_{j+1}) \mapsto (-\overline{z_{j+1}}, \overline{z_1}, \dots, \overline{z_{j}}).
	\end{equation}
	Clearly, $\rho_j^\ell$ contains a non-trivial permutation of the coordinates $(z_1,\dots, z_{j+1})$ unless $\ell \equiv 0 \mod{(j+1)}$. Furthermore, when $j+1$ is even, there holds for $z \in \CC^{j+1}$,
	\[
	\rho^{j+1}_{j}(z) = -z \quad \text{and} \quad \rho_j^{2(j+1)}(z) = z
	\]
	On the other hand, it is easily seen that if $j+1$ is odd:
	\[
	\rho^{j+1}_j(z_1,\dots,z_{j+1}) = (-\bar{z_1}, -\bar{z_2},\dots, - \overline{z_{j+1}}) \quad \text{and} \quad \rho^{2(j+1)}_j(z) = z.
	\]
	Regardless, $\rho_j$ has order $2(j+1)$ for each $j \ge 1$. Additionally, we allow rotations to act upon the coordinates of $\CC^{j+1}$ in the synchronous coordinate-wise manner:
	\[
	e^{i\theta}(z_1,\dots,z_{j+1}) := (e^{i\theta} z_1,\dots,e^{i\theta}z_{j+1}),
	\]
	for $\theta \in [0,2\pi)$. Define $\Gamma_j$ to be the subgroup of $O(2(j+1))$, acting upon the coordinates of $\RR^{2(j+1)} \cong \CC^{j+1}$, generated by these rotations and the map $\rho_j$. We emphasize that, when $j=1$, this idea comes from Clapp \cite{Clapp_pure} and Clapp-Rios \cite{clappRios}.
	
	An easy calculation shows that, for each $\theta \in [0,2\pi)$,
	\begin{align*}
		e^{i\theta} \rho_j(z_1,\dots,z_{j+1}) = e^{i{\theta}} (-\overline{z_{j+1}}, \bar{z_1},\dots, \overline{z_{j}}) &= (-e^{i\theta}\overline{z_{j+1}},e^{i\theta} \bar{z_1},\dots, e^{i\theta}\overline{z_{jj}}) \\
		&= (-\overline{e^{-i\theta} z_{j+1}}, \overline{e^{-i\theta} z_1},\dots, \overline{e^{-i\theta} z_{j}})\\
		&= \rho_j e^{i(-\theta)}(z_1,\dots,z_{j+1})\\
		&= \rho_j e^{i(2\pi - \theta)}(z_1,\dots,z_{j+1}).
	\end{align*}
	Consequently, rotations permute with $\rho_j$ up to a change in angle in the sense that $e^{i\theta} \rho_j = \rho_j e^{i(2\pi - \theta)}$. In the event that $j+1$ is even, we have $\rho_j^{j+1}(z) = -z$ whence $\rho_j^{j+1} = e^{i\pi}$ is a rotation. However, for all $(j+1)$-odd, no composition of $\rho_j$ can be realized as a rotation (except the identity). It follows from these observations that $\Gamma_j$ consists precisely of those linear isometries $g \in O(2(j+1))$ uniquely expressible in the form
	\begin{equation}\label{eq:uniqueExpression}
		\begin{cases}
			g = \rho_j^{\epsilon} e^{i\theta}, \quad \epsilon \in \{0,1, \dots, j\}, \, \theta \in [0,2\pi) & \text{if $j+1$ is even},\\
			g = \rho_j^{\epsilon} e^{i\theta}, \quad \epsilon \in \{0,1,\dots, 2j+1\},\, \theta \in [0,2\pi) & \text{if $j +1$ is odd}.
		\end{cases}
	\end{equation}
	This representation allows us to define a mapping
	\[
	\phi_j : \Gamma_j \to \{\pm 1\}, \quad g \mapsto (-1)^\epsilon,
	\]
	where $g$ is expressed uniquely as in \eqref{eq:uniqueExpression}; clearly, this is a group homomorphism. Moreover, $\phi_j$ is continuous being that it is a homomorphism of topological groups that is continuous at the identity. Indeed, consider $\eta_j := (1+i, 1, 0\dots, 0) \in \CC^{j+1}$ and write $g \in \Gamma_j$ according to \eqref{eq:uniqueExpression}. If $(j+1)$ is even and $\epsilon>0$, then $\abs{g\eta_j - \eta_j} \ge 1$ implying that any $g$ sufficiently close to the identity belongs to the kernel of $\phi_j$. Similarly, when $(j+1)$ is odd, we clearly have $\abs{g\eta_j - \eta_j} \ge 1$ unless $\epsilon = 0$ or $\epsilon = j+1$. When $\epsilon = j+1$, direct computation shows that \[\abs{g\eta_j - \eta_j}^2 = 6-4\sin(\theta)+2\cos(\theta) \ge 6 - 2\sqrt{5}.\] Thus, again in this case, any $g \in \Gamma_j$ sufficiently close to the identity lies in the kernel of $\phi_j$. This demonstrates the continuity of $\phi_j$ on $\Gamma_j$.

	\subsubsection{Incorporating Pinwheel-type Symmetries}\label{section:genInfinity}
	
	In the following, we borrow ideas from Clapp-Vicente-Ben\'itez \cite{clapp2026entire} (see also Clapp-Faya-Saldan\~a \cite{Clapp_2024}, and Clapp-Salda\~na-Soares-Ben\'itez \cite{clapp2025optimal})
	in order to produce solutions satisfying a ``pinwheel-type" symmetry within the weighted case. We remark that whether or not we incorporate such a symmetry is encoded by the value of $\alpha$ from Theorem \ref{thm:main}. More precisely, $\alpha$ will determines the number of ``pins'' within the symmetry class of our solution. 
	
	As in the paper of Clapp-Vicente-Ben\'itez \cite{clapp2026entire}, we begin by considering an action on $\CC^2 \cong \RR^4$. We allow the unit circle to act upon $\CC^2$ via asynchronous rotation:
	\[
	e^{i\theta}(z_1,z_2) := (e^{i\theta}z_1, e^{-i\theta}z_2).
	\]
	Next, given a non-negative integer $\alpha$ and $\beta \in \NN$, we adopt the transformation 
	\[
	\varrho_\alpha^\beta : \CC^2 \to \CC^2, \quad	
	z \mapsto \cos\left( \frac{\beta\pi}{2^{\alpha+1}} \right)z + \sin\left( \frac{\beta\pi}{2^{\alpha+1}} \right)\rho_1(z).
	\]
	For convenience, we write $\varrho_\alpha := \varrho_\alpha^1$. As illustrated in Clapp-Vicente-Ben\'itez \cite{clapp2026entire}, $\varrho_\alpha$ commutes with these rotations and, moreover, $\varrho_\alpha$ has order $2^{\alpha+2}$. In fact, there holds $\varrho_\alpha^{\beta}\varrho_\alpha^{\beta^\prime} = \varrho_\alpha^{\beta+\beta^\prime}$. If $\alpha = 0$, let
	\[
	\Upsilon_\alpha = \{I\}
	\]
	where $I$ denotes the identity of $O(n)$. Otherwise, let $\Upsilon_\alpha$ denote the copy in $O(n)$ of the $O(4)$-subgroup generated by $\varrho_\alpha$ and these rotations.  Next, by the relations above, the mapping
	\[
	\varphi_\alpha : \Upsilon_\alpha \to \{ \pm 1 \}
	\]
	induced by the rules $\varphi_\alpha(e^{i\theta}) := 1$ and $\varphi_\alpha(\varrho_\alpha) := -1$ is a well defined group homomorphism. If $\alpha = 0$, we let $\varphi_\alpha$ be the trivial homomorphism. Notice (see, for instance, Clapp-Vicente-Ben\'itez \cite{clapp2026entire}) that $\Upsilon_\alpha$ is compact in $O(n)$ and $\varphi_\alpha$ is continuous. Furthermore, it is easy to see that the point $(1,0) \in \CC^2$ satisfies the condition $\operatorname{Stab}_{\Upsilon_\alpha}(1,0) \subseteq \ker{\varphi_\alpha}$.
	
	\subsection{Group Construction}
	
	Fix now a dimension $n \ge 4$ and a non-negative integer $\alpha$. Suppose we are given a tuple $\m := (m_1,\dots, m_k)$ of non-negative integers and define the following conditions:
	
	Assume henceforth that \eqref{eq:primeCondition} is satisfied.
	For each index $j = 1,\dots, k$, we denote by $\Gamma_{j,\ell}$, with $\ell = 1,\dots, m_j$, the copy of $\Gamma_j$ in $O(n)$ acting upon the coordinates of $\RR^n$ indexed by
	\begin{align*}
		\begin{cases}
			\sum_{\iota=1}^{j-1} 2m_\iota (\iota+1) + 2(\ell-1)(j+1) + 5 & \text{if $\alpha > 0$,}\\
			\sum_{\iota=1}^{j-1} 2m_\iota (\iota+1) + 2(\ell-1)(j+1) + 1 & \text{if $\alpha = 0$,}
		\end{cases}
	\end{align*}
	through
	\begin{align*}
		\begin{cases}
			\sum_{\iota=1}^{j-1} 2m_\iota (\iota+1) + 2\ell (j+1) +4 & \text{if $\alpha > 0$,}\\
			\sum_{\iota=1}^{j-1} 2m_\iota (\iota+1) + 2\ell (j+1) & \text{if $\alpha = 0$.}
		\end{cases}
	\end{align*}
	Since $\Gamma_{j,\ell}$ is an isomorphic copy of $\Gamma_j$, we may abuse notation and view $\phi_j$ as a group homomorphism $\Gamma_{j,\ell} \to \{\pm 1\}$ in the natural way. 
	
	According to this scheme, given any non-negative integer $\alpha$ and $\m$ such that $(\alpha,\m)$ satisfies \eqref{eq:primeCondition}, we define a subgroup $G_{\alpha,\m}$ of $O(n)$ by putting 
	\begin{equation}\label{eq:primeGroupDef}
		G_{\alpha, \m} := \Upsilon_\alpha\Lambda_{\alpha,\m}\prod_{j=1}^k \prod_{\ell=1}^{m_j} \Gamma_{j,\ell},
	\end{equation}
	where $\Lambda_{\alpha,\m}$ is a copy of the orthogonal group:
	\[
	\Lambda_{\alpha,\m} \cong \begin{cases}
		O\left( n - 4 - 2\sum_{j=1}^k m_j(j+1)\right) & \text{if $\alpha > 0$ and \eqref{eq:primeCondition2} holds},\\
		O\left( n -  2\sum_{j=1}^k m_j(j+1)\right) & \text{if $\alpha = 0$ and \eqref{eq:primeCondition2} holds},\\
		\{I\} & \text{if \eqref{eq:primeCondition2} fails}
	\end{cases}
	\] acting on the remaining coordinates of $\RR^n$, untouched by all of the $\Gamma_{j,\ell}$ and $\Upsilon_\alpha$. Observe that $G_{\alpha,\m}$ is indeed a well defined subgroup of $O(n)$ since it is the product of subgroups of $O(n)$, with each subgroup commuting with the others.
	
	It is apparent from the definition in \eqref{eq:primeGroupDef} that each $g \in G_{\alpha,\m}$ may be uniquely expressed as a product of commutative factors belonging to the respective subgroups appearing in the definition of $G_{\alpha,\m}$, in the sense that
	\begin{equation}\label{eq:elementRepresentation}
		g = \upsilon_\alpha g_0 \prod_{j=1}^k \prod_{\ell=1}^{m_j} \gamma_{j,\ell}, \quad \upsilon_\alpha \in \Upsilon_\alpha,\,g_0 \in \Lambda_{\alpha,\m},\, \gamma_{j,\ell} \in \Gamma_{j,\ell}.
	\end{equation}
	Consequently, we define a group homomorphism $\phi_{\alpha,\m} : G_{\alpha,\m} \to \{\pm 1\}$ by allowing each $\phi_j$ to act on the factors of $g$ belonging to the copies of $\Gamma_j$, i.e. 
	\begin{align*}
		\phi_{\alpha,\m}(g) :=  \varphi_\alpha(\upsilon_\alpha)\prod_{j=1}^k \prod_{\ell=1}^{m_j} \phi_j(\gamma_{j,\ell}),
	\end{align*}
	where $g$ is expressed as in \eqref{eq:elementRepresentation}.
	
	Before we proceed with some essential lemmas regarding the properties of the pairs $(G_{\alpha,\m}, \phi_{\alpha,\m})$ constructed above, let us formally encode some notation which is utilized in the statement Theorem \ref{thm:main}.
	
	\begin{defn}\label{def:Sn}
		Maintaining the aforedescribed notation, a function $u : \RR^n \to \RR$ will be said to be $(\alpha,\m)$-symmetric if it is $\phi_{\alpha,\m}$-equivariant.
	\end{defn}
	
	Roughly speaking, an $\m$-type symmetry describes in which ways the solution $u$, with respect to the remaining coordinates not affected by $\Upsilon_\alpha$, may be synchronously rotated, reflected, and coordinate swapped in relation to the function's sign. Furthermore, as we shall later demonstrate, these indeed give rise to distinct and incompatible symmetries, in contexts where the assumptions stated in the second part of Theorem \ref{thm:main} hold.
	
	\begin{rem}\label{rem:counterExample}
		It is natural to ask whether one must impose some condition on $\m$ and $\n$, as in Theorem \ref{thm:main}-(ii), to distinguish the solutions when $\alpha = \beta$.  Without any additional condition, there is no way, in general, to distinguish these solutions using the symmetries alone. For example, the mapping
		\[
		\RR^4 \cong \CC^2 \to \RR, \quad z \mapsto \Im\left( z_1\overline{z_2}z_3\overline{z_4} \right)
		\]
		is an example of a non-trivial function which is both $(0,(2,\mathbf{0}))$-symmetric and $(0,(0,0,1,\mathbf{0}))$-symmetric. Consequently, if we allow for actions on a composite number of complex coordinates, there is no way, in general, to distinguish the resulting symmetry types. In addition, to show that our imposed symmetries are mutually exclusive, we utilize a binary coding scheme together with a $\gcd$-argument requiring that we compare symmetries on a coprime number coordinates. We refer the reader to Lemmas \ref{lem:euclid} and \ref{lem:isCode}, as well as Proposition \ref{prop:distinct} for the concrete arguments.
	\end{rem}
	
	Next, we show that the group and homomorphism pairs constructed above satisfy (P1)-(P2) whenever condition \eqref{eq:primeCondition} holds. In addition, we verify that \eqref{eq:primeCondition2} implies property (P3). 
	
	\begin{lem}\label{lem:primeConstruct}
		Let $n \ge 4$, $\alpha \ge 0$ an integer, and $\m := (m_1,\dots, m_k)$ be a tuple of non-negative integers such that $(\alpha, \m)$ satisfies \eqref{eq:primeCondition}. Then, the group $G_{\alpha,\m}$ defined in \eqref{eq:primeGroupDef} is a closed subgroup of $O(n)$, and $(G_{\alpha,\m}, \phi_{\alpha,\m})$ satisfies conditions (P1)-(P2). Furthermore, if $(\alpha,\m)$ satisfies \eqref{eq:primeCondition2}, then (P3) holds true.
	\end{lem}
	
		\begin{proof}
		We only explicitly treat the case $\alpha > 0$, since a direct adaptation of this argument applies in the special case $\alpha = 0$ (see also the end of this proof for comments on how the proof is adapted in this case).
		
		First we demonstrate that $G_{\alpha,\m}$ is a closed subgroup of $O(n)$. Indeed, owing to the fact that each $\rho_j$ has finite order, and the set of actions by rotation form a continuous copy of the unit circle in $O(2(j+1))$, it is easy to see that $\Gamma_j$ is compact in $O(2(j+1))$. It follows that each $\Gamma_{j,\ell}$ is compact in $O(n)$. By Tychonoff's theorem and continuity of the group operation on $O(n)$, we infer that $\Upsilon_\alpha\prod_{j=1}^k \prod_{\ell=1}^{m_j} \Gamma_{j,\ell}$ is compact in $O(n)$. Then, since $\Lambda_{\alpha,\m}$ is closed in $O(n)$, it follows from standard theory on topological groups (see Hewitt-Ross \cite{hewittRoss}) that $G_{\alpha,\m}$ is closed in $O(n)$, which verifies our first assertion.

		By construction, $\phi_{\alpha,\m}$ is surjective and hence criterion (P2) holds true. In addition, it is easy to see that $\phi_{\alpha,\m}$ is continuous on $G_{\alpha,\m}$ since, by our discussions in \S2.1-\S2.1.1, $\varphi_\alpha$ and each $\phi_j$ are continuous. To see that the property (P1) holds, let $\mathbf{e} = (1,0) \in \CC^2$ and define for each index $j=1,\dots,k$ the point $\xi_{j} = (1+i,1,\dots,1) \in \CC^{j+1}$; consider 
		\begingroup
		\small
		\[
		\xi = (\mathbf{e}, \xi_{1}, \dots, \xi_{1}, \dots, \xi_{k}, \dots, \xi_{k}, w) \in \CC^2 \times\left(\prod_{j=1}^k\CC^{m_j(j+1)}\right)\times \RR^{n-4-2\sum_{j=1}^k m_j(j+1)}
		\]
		\endgroup
		with $w \in \RR^{n-4-2\sum_{j=1}^k m_j (j+1)}$ arbitrary. Above, we emphasize that each $\xi_{j}$ appears exactly $m_j$-times in the tuple defining $\xi$. We claim that $\operatorname{Stab}_{G_{\alpha,\m}}(\xi) \subseteq \ker{\phi_{\alpha,\m}}$. Certainly, express any  $g \in \operatorname{Stab}_{G_{\alpha,\m}}(\xi) $ according to \eqref{eq:elementRepresentation} and observe that $g \xi = \xi$ forces $\upsilon_\alpha \mathbf{e} = \mathbf{e}$ and $\gamma_{j,\ell} \xi_{j} = \xi_{j}$ for all $j  = 1,\dots, k$ and $\ell=1,\dots, m_j$. By our comments at the end of \S\ref{section:genInfinity}, we have $\upsilon_\alpha \in \ker{\varphi_\alpha}$. For each admissible pair $j,\ell$, write $\gamma_{j,\ell}$ in the form of \eqref{eq:uniqueExpression}. That is, $\gamma_{j,\ell} = \rho_{j,\ell}^{\epsilon_{j,\ell}} e^{i\theta_{j,\ell}}$ with $\epsilon_{j,\ell} \in \{0,1, \dots, j\}$ if $j+1$ is even and $\epsilon_{j,\ell} \in \{0,1,\dots,2j+1\}$ otherwise. Here, $\rho_{j,\ell}$ denotes the copy of $\rho_j$ acting upon the coordinates treated by $\Gamma_{j,\ell} \cong \Gamma_j$. Note that, because $\rho_{j,\ell}^{\epsilon_{j,\ell}}$ shuffles coordinates unless $\epsilon_{j,\ell} \equiv 0 \mod{(j+1)}$, we must have either $\epsilon_{j,\ell} = 0$ or $\epsilon_{j,\ell} = j+1$. If $j+1$ is even, the only possibility is $\epsilon_{j,\ell} = 0$. If $j+1$ is odd, then in the event that $\epsilon_{j,\ell} = j+1$, the equality $\rho_{j,\ell}^{\epsilon_{j,\ell}} e^{i\theta_{j,\ell}} \xi_{j} = \xi_{j}$ reduces to 
		\[
		\left( -e^{-i\theta_{j,\ell}} (1-i), -e^{-i\theta_{j,\ell}},-e^{-i\theta_{j,\ell}},\dots,-e^{-i\theta_{j,\ell}} \right) = (1+i,1,1,\dots,1)
		\]
		which admits no possible solutions. Consequently, each $\epsilon_{j,\ell} =0$ whence $g \in \ker{\phi_{\alpha,\m}}$ by construction. This verifies (P1). 
		
		It remains to establish (P3) when $(\alpha,\m)$ satisfies the additional condition \eqref{eq:primeCondition2}. To this end, notice that for any 
		\begingroup
		\smaller
		\[x = (z, \zeta_{1,1},\dots,\zeta_{1,m_1}, \dots, \zeta_{k,1}, \dots, \zeta_{k,m_k}, w)
		\in \CC^2 \times\left(\prod_{j=1}^k\CC^{m_j(j+1)}\right)\times \RR^{n-4-2\sum_{j=1}^k m_j(j+1)} 
		\]
		\endgroup
		there holds
		\begin{equation}\label{eq:orbitForm}
			\operatorname{Orb}_{G_{\alpha,\m}}(x) = \Upsilon_\alpha z \times \left(\prod_{j=1}^k \prod_{\ell=1}^{m_j} \Gamma_{j,\ell} \zeta_{j,\ell}\right) \times \Lambda_{\alpha,\m} w.
		\end{equation}
		Thus, so long as $x \ne 0$, the cardinality of $\operatorname{Orb}_G(x)$ shall be infinite because condition \eqref{eq:primeCondition2} ensures that $\Lambda_{\alpha,\m}$ is infinite in the case $n \ne 4 + 2\sum_{j=1}^k m_j(j+1)$. This demonstrates (P3) and concludes the proof.
		
		In the case $\alpha =0$, we proceed in the same way as above. Indeed, by the same reasoning as before, for $\alpha =0$ and $\m$ satisfying \eqref{eq:primeCondition}, the group $G_{\alpha,\m}$ given by  \eqref{eq:primeGroupDef} defines a closed subgroup of $O(n)$. The corresponding homomorphism is obviously surjective by construction. To verify the stabilizer condition (P1), we use the same argument as in the case $\alpha = 0$, but omitting the point $\mathbf{e} \in \CC^2$ and only including the required copies of $\xi_j \in \CC^{j+1}$ within the definition of $\xi$. If $(\alpha,\m)$ satisfies \eqref{eq:primeCondition}-\eqref{eq:primeCondition2} with $\alpha=0$, then the orbit of a point $x = (\zeta_{1,1},\dots,\zeta_{1,m_1}, \dots, \zeta_{k,1}, \dots, \zeta_{k,m_k}, w)$ is given by
		\[ 	\operatorname{Orb}_{G_{\alpha,\m}}(x) = \left(\prod_{j=1}^k \prod_{\ell=1}^{m_j} \Gamma_{j,\ell} \zeta_{j,\ell}\right) \times \Lambda_{\alpha,\m} w
		\]
		and hence (P3) holds true by the same logic as for the case $\alpha > 0$.
	\end{proof}
	
	\subsection{The Distinguishing of Symmetry Configurations}
	We now turn towards the task of investigating when the symmetry types $(\alpha,\m)$ satisfying \eqref{eq:primeCondition} are incompatible. To achieve this, we leverage a structure which we call a $t$-code.
	
	For any positive integer $t$, given a binary vector $c := (c_1,\dots,c_t) \in \ZZ_2^t$, we shall write $\tilde{c}$ to denote the \emph{cycle} of $c$ given by
	\[
	\tilde{c} := (c_t, c_1,\dots,c_{t-1}).
	\]
	
	\begin{defn}
		A $t$-code is a subset $\C$ of the vector space $\ZZ_2^t$ such that
		\begin{enumerate}		
			\item If $c, c^\prime\in C$ with $c\le c^\prime$ (meaning $c_j \le c_j^\prime$ for all $j=1, \dots, t$) then $c+c^\prime \in \C$. 
			\item $\tilde{c} \in \C$ for each $c \in \C$.
		\end{enumerate} 
		An element $c \in \C$ is called a \emph{codeword} (of $\C$).
	\end{defn}
	
	Given an arbitrary integer $0\le r \le t$, it will be convenient to associate it with the codeword
	\[
	v_r := (1,1,\dots,1,0,0,\dots,0)\in \ZZ_2^t,
	\]
	whose first $r$-coordinates are $1$ and the remaining $(t-r)$-coordinates are $0$. Note that $v_0$ is simply the zero codeword while $v_t$ is the codeword consisting entirely of ones.

	Before presenting the main lemma of this subsection, we make a simple observation. If $r \leq s \le t$ are non-negative integers and $v_r, v_s \in \C$ where $\C$ is a $t$-code, then $v_{s-r} \in \C$. Indeed, if $v_r, v_s\in \C$ then $v_r+v_s\in \C$ and repeatedly cycling $v_r+v_s$ yields $v_{s-r} \in \C$.

	\begin{lem}\label{lem:euclid}
		Let $r<s\le t$ be positive integers such that $\gcd(r,s) = 1$. Suppose that $\C$ is a $t$-code containing both $v_r$ and $v_s$. Then, $\C$ contains the standard basis of $\ZZ_2^t$.
	\end{lem}
	\begin{proof}
		Since $\C$ is closed under cycling, it suffices to show that $v_1 = (1, 0, \dots, 0) \in \C$. If $r = 1$ there is nothing to show. Otherwise, we write
		\[
		s = k_1 r + r_1
		\]
		for a non-negative integer $k_1$ and an integer remainder $0< r_1 < r$. Then, since $v_r, v_s \in \C$ we see that $v_{s-k_1r} = v_{r_1} \in \C$. If $r_1 = 1$ then we are done. Otherwise, we continue and write
		\[
		r = k_2r_1 + r_2
		\]
		for a non-negative integer $k_2$ and an integer remainder $0< r_2 < r_1$. Since $v_r\in \C$ and $v_{r_1} \in \C$, we also have that $v_{r-k_2r_1} = v_{r_2} \in \C$. Again, if $r_2 = 1$ then we are done. Otherwise, continuing  this process, since $\gcd(r,s) = 1$ the Euclidean algorithm guarantees that, eventually, we have a remainder of $1$. So $v_1\in \C$ and we are done.
	\end{proof}
	
	\subsection{Rotation Codewords}
	
	Keeping with the notation presented in this section, let $\alpha \ge 0$ be an integer and $\m$ a tuple such that $(\alpha,\m)$ satisfies \eqref{eq:primeCondition}. Let $1 \le j \le k$ be such that $m_j > 0$. Fix an index $\ell \in \{1,\dots,m_j\}$ and denote by $(z_1, z_2, \dots, z_{j+1})$ the (complex) coordinates the group $\Gamma_{j,\ell} \subseteq G_{\alpha,\m}$ acts upon. Then, any $x\in \mathbb{R}^n$ may be expressed as $x = (w_1, z_1, \dots, z_{j+1}, w_2)$, where $w_1$ and $w_2$ are real tuples containing, respectively, the coordinates of $x$ before and after $(z_1, \dots, z_{j+1})$.
	
	Given a binary codeword $c = (c_1, \dots, c_{j+1}) \in \ZZ_2^{j+1}$ of length $(j+1)$, we define a rotation operator 
	\[
	R^{\alpha,\m, j,\ell}_{c}(\theta) : \RR^n \to \RR^n, \;\; (w_1, z_1, \dots, z_{j+1}, w_2) \mapsto (w_1, e^{c_1i\theta}z_1, \dots, e^{c_{j+1}i\theta}z_{j+1}, w_2).
	\]
	The key observation is that the set of all $c \in \ZZ_2^{j+1}$ such that $u$ is invariant under $R^{\alpha,\m, j,\ell}_{c}(\theta)$ is itself a $(j+1)$-code.
	\begin{lem}\label{lem:isCode}
		Let $n \ge 4$ and $(\alpha,\m)$ be a pair for which \eqref{eq:primeCondition} holds. Suppose $1 \le j \le k$ is an index such that $m_j >0$ and fix $\ell \in \{1,\dots, m_j\}$. Assume $u$ is $(\alpha,\m)$-symmetric and let $\C$ be the set of all binary codewords $c \in \ZZ_2^{j+1}$ such that
		\[
		u(R_c^{\alpha,\m, j,\ell}(\theta)x) = u(x)
		\]
		for all $x \in \RR^n$ and $\theta \in [0,2\pi)$. Then, $\C$ is a $(j+1)$-code.
	\end{lem}
	\begin{proof}
		For concision, let us write $R_c(\theta) := R^{\alpha,\m, j,\ell}_{c}(\theta)$ for the duration of the proof. 
		Clearly, $\C$ contains $(0,0,\dots, 0)$ by definition. If $c, c^\prime \in \C$ with $c \leq c^\prime$ then $c+c^\prime \in \C$ since
		\begin{align*}
			R_c(-\theta)R_{c^\prime}(\theta)(w_1,z_1,\dots, z_{j+1}, w_2)
			&= (w_1,e^{(c^\prime_1-c_1)i\theta}z_1,\dots, e^{(c^\prime_{j+1}-c_{j+1})i\theta}z_{j+1}, w_2)\\
			&= R_{c+c^\prime}(\theta)(w_1,z_1,\dots, z_{j+1}, w_2)
		\end{align*}
		where $c+c^\prime$ is computed in $\ZZ_2^{j+1}$. Next, to show $\C$ is closed under cycling, we write
		\begin{align*}
			\rho_j R_c(-\theta)\rho_{j}^{-1}(w_1,z_1,\dots, z_{j+1}, w_2) &= \rho_j R_c(-\theta)(w_1,\overline{z_2},\dots, \overline{z_{j+1}}, -\overline{z_1}, w_2)\\
			&= \rho_j (w_1, e^{-i\theta c_1}\overline{z_2},\dots, e^{-i\theta c_{j}}\overline{z_{j+1}}, -e^{-i\theta c_{j+1}}\overline{z_1}, w_2)\\
			&= (w_1, e^{i\theta c_{j+1}} z_1, e^{i\theta c_1}z_2, \dots, e^{i\theta c_{j}} z_{j+1}, w_2)\\
			&= R_{\tilde{c}}(\theta)x.
		\end{align*}
		In the above, $\rho_j = \rho_{j,\ell}$ is the copy of $\rho_j \in \Gamma_j$ appearing in $\Gamma_{j,\ell}$. Thus, $\tilde{c} \in \C$ and we are done.
	\end{proof}
	
	\begin{prop}\label{prop:distinct}
		Fix a dimension $n \ge 4$. Let $(\alpha, \m)$ and $(\beta, \n)$ be pairs such that each satisfies condition \eqref{eq:primeCondition}. Assume that $u : \RR^n \to \RR$ is $(\alpha,\m)$-symmetric and $(\beta,\n)$-symmetric. If $\alpha = \beta$ and $\m \ne \n$, then $u$ vanishes identically if $\m \lesssim \n$ or $\gcd(\m,\n) = 1$.
	\end{prop}
	\begin{proof}
		We may assume without harm to the proof that $\alpha = \beta = 0$ (otherwise we consider only the coordinates $(x_5,x_6,\dots, x_n)$ of $x$). By hypothesis, $\m \ne \n$. Thus, there exists a minimal index $j \in \{1,\dots, k\}$ such that $n_j \ne m_j$. Without loss of generality, suppose we are in the case $m_j < n_j$. We now distinguish two possible settings.
		\begin{enumerate}[leftmargin=*]
			\item If $m_\ell = 0$ for all indices $\ell > j$ then, owing to the fact that $m_j < n_j$, there exists a copy of $\rho_j$ acting on $2(j+1)$-coordinates such that
			\[
			u(\rho_j x) = u(x) \quad \text{and} \quad u(\rho_jx) = -u(x), \quad \forall x \in \RR^n,
			\]
			where the first identity arises from the $\phi_{\alpha,\m}$-equivariance of $u$ and the second comes from the $\phi_{\beta,\n}$-equivariance. We conclude that $u \equiv 0$.
			\item Otherwise, let $\ell > j$ be the minimal index for which $m_\ell > 0$. In this case, there exists a tuple consisting of $2(\ell+1)$-coordinates on which a copy of $\Gamma_{\ell}$ acts, but for which the first $2(j+1)$-coordinates are acted upon by a copy of $\Gamma_{j}$. 
			
			Let now $\C$ denote the set of all binary codewords $c \in \ZZ_2^{\ell+1}$ such that 
			\[
			u(R_c^{\alpha,\m, \ell, 1}(\theta)x) = u(x), \quad \forall \theta \in [0,2\pi)
			\]
			and all $x \in \RR^n$. By Lemma \ref{lem:isCode}, $\C$ is a $(\ell+1)$-code. From the isomorphic copies of $\Gamma_{\ell}$ and $\Gamma_j$, it is automatic that $v_{j+1}, v_{\ell+1} \in \C$. Observe that the condition $\gcd(\m,\n) = 1$ enforces $\gcd(\ell+1, j+1) = 1$. An application of Lemma \ref{lem:euclid} implies that $\C$ contains the standard basis of $\ZZ_2^{\ell+1}$, thereby ensuring that $u$ is invariant under rotations in each individual coordinate acted upon by the aforementioned copy of $\Gamma_{\ell}$.
			
			By the aforementioned $\gcd$-condition, at least one of $\ell+1$ or $j+1$ must be odd. Thus, we are reduced to two cases.
			\begin{enumerate}[label=(\roman*)]
				\item If $(\ell+1)$ is odd, then
				\[
				\begin{aligned}
					u(w_1, z_1, \dots, z_{\ell+1}, w_2)
					&= - u(\rho_\ell^{\ell+1} (w_1, z_1, \dots, z_{\ell+1}, w_2)) \\
					&= -u(w_1, -\overline{z_1}, \dots, -\overline{z_{\ell+1}}, w_2)
				\end{aligned}
				\]
				for any $(w_1, z_1, \dots, z_{\ell+1}, w_2) \in \RR^n$. But then, since $u$ is invariant under rotations in each of the (complex) coordinates $z_1, \dots, z_{\ell + 1}$, it follows that
				\[
				u(w_1, -\overline{z_1}, \dots, -\overline{z_{\ell+1}}, w_2) = u(w_1, z_1, \dots, z_{\ell+1}, w_2).
				\]
				Combining the last two equations, we get
				\[
				u(w_1, z_1, \dots, z_{\ell+1}, w_2) = -u(w_1, z_1, \dots, z_{\ell+1}, w_2).
				\]
				We conclude that $u\equiv 0$.
				\item If $(j+1)$ is odd then, by the same argument,
				\[
				\begin{aligned}
					u(w_1, z_1, \dots, z_{j+1}, w_2)
					&= - u(\rho_j^{j+1} (w_1, z_1, \dots, z_{j+1}, w_2)) \\
					&= -u(w_1, -\overline{z_1}, \dots, -\overline{z_{j+1}}, w_2)
				\end{aligned}
				\]
				for any $(w_1, z_1, \dots, z_{\ell+1}, w_2) \in \RR^n$. Using that  $u$ is rotation invariant in each of the coordinates  $z_1, \dots, z_{j+1}, \dots, z_{\ell + 1}$, we again find that
				\[
				u(w_1, -\overline{z_1}, \dots, -\overline{z_{j+1}}, w_2) = u(w_1, z_1, \dots, z_{j+1}, w_2).
				\]
				As before, we obtain
				\[
				u(w_1, z_1, \dots, z_{j+1}, w_2) = -u(w_1, z_1, \dots, z_{j+1}, w_2).
				\]
				Ergo, $u\equiv 0$.
			\end{enumerate}
			In any case, we deduce that $u \equiv 0$ which completes the proof.
		\end{enumerate}
		
	\end{proof}

	\begin{prop}\label{prop:distinct2}
		Let $n \ge 4$ and suppose $(\alpha, \m)$ and $(\beta, \n)$ are two tuples, each satisfying \eqref{eq:primeCondition}. Let $u : \RR^n \to \RR$ be $(\alpha,\m)$-symmetric and $(\beta,\n)$-symmetric. If $\alpha \ne \beta$, then $u$ vanishes identically.
	\end{prop}
	\begin{proof}
		Since $\alpha \ne \beta$, we may assume without loss of generality that $\alpha < \beta$. In this case, we must distinguish between $\alpha = 0$ and $\alpha > 0$. In the latter case, we have 
		\(
		\varrho_\alpha = \varrho_{\beta}^{2^{\beta-\alpha}}
		\)
		implying that, for all $x \in \RR^n$, 
		\begin{align*}
			u(x) = u(\varrho_{\beta}^{2^{\beta-\alpha}}(x))
			= u(\varrho_\alpha(x))
			&= -u(x),
		\end{align*}
		where the first equality uses the $(\beta,\n)$-symmetry of $u$ and the final equality invokes the $(\alpha,\m)$-symmetry. Consequently, we infer that $u \equiv 0$.
		
		If instead $\alpha = 0$, then observing that $\varrho_{\beta}^{2^{\beta}} = \rho_1$ yields a similar contradiction provided $m_1 \ne 0$. Indeed,
		\begin{align*}
			u(x) = u(\varrho_{\beta}^{2^{\beta}}(x))
			= u(\rho_1(x))
			&= -u(x)
		\end{align*}
		whence $u \equiv 0$.
		It remains to treat the setting where $0 = \alpha < \beta$ and $m_1 = 0$. In this case,  let $j > 1$ be minimal such that $m_j \ne 0$. Denote by $(z_1,z_2,\dots, z_{j+1})$ the complex coordinates acted upon by the first copy of $\Gamma_j$ in $G_{\alpha, \m}$. Then, for any $x = (z_1,\dots, z_{j+1}, w) \in \RR^n \cong \CC^{j+1} \times \RR^{n-2(j+1)}$, there holds
		\begin{equation}\label{eq:rotateAll}
			u(x) = u(e^{i\theta}z_1, \dots, e^{i\theta} z_{j+1}, w)
		\end{equation}
		for any $\theta \in \RR$, by virtue of the $(\alpha,\m)$-symmetry. Owing to the $(\beta, \n)$-symmetry, we also have
		\begin{align*}
			u(x) = u(e^{i\theta}z_1, e^{-i\theta} z_{2}, z_3, \dots, z_{j+1}, w)
		\end{align*}
		for all angles $\theta \in \RR$. By following the argument in Lemma \ref{lem:isCode} (i.e. conjugating the above rotations by the copy of $\rho_j$ in $G_{\alpha,\m}$ acting on the coordinates $(z_1,\dots,z_{j+1})$), we see that 
		\begin{align}
			u(x) &= u(e^{i\theta}z_1, e^{-i\theta} z_{2}, z_3, \dots, z_{j}, z_{j+1}, w) \tag{E.1}\\
			&= u(z_1, e^{i\theta} z_{2}, e^{-i\theta}z_3, z_4, \dots, z_{j}, z_{j+1}, w)\tag{E.2}\\
			&= u(z_1, z_2, e^{i\theta} z_{3}, e^{-i\theta}z_4, \dots, z_{j}, z_{j+1}, w)\tag{E.3}\\
			&\,\,\,\vdots\tag*{}\\
			&= u(z_1, z_2, z_{3}, z_4, \dots, e^{i\theta}z_{j}, e^{-i\theta}z_{j+1}, w) \tag{E.j}
		\end{align}
		Now, given $\theta$, we sequentially apply (E.1) with $\theta/(j+1)$, (E.2) with $2\theta/(j+1)$, (E.3)  with $3\theta/(j+1)$, and so forth, until we apply (E.j) with $j\theta/(j+1)$. Each such rotation preserves the sign of $u$, whence
		\[
		u(x) = u(e^{i\theta/(j+1)}z_1, e^{i\theta/(j+1)} z_2, \dots, e^{i\theta/(j+1)} z_{j}, e^{-ij\theta/(j+1)}z_{j+1}, w).
		\]
		Then, applying \eqref{eq:rotateAll} with $-\theta/(j+1)$, we obtain
		\[
		u(x) = u(z_1, z_2, \dots, z_{j}, e^{-i\theta}z_{j+1}, w).
		\]
		Consequently, $u$ is rotation invariant in the complex coordinate $z_{j+1}$. Following once more the conjugation argument from Lemma \ref{lem:isCode}, it follows that $u$ is rotation invariant in each complex coordinate of $(z_1,\dots, z_{j+1})$. 
		
		Next, by using that $u$ is $(\beta,\n)$-symmetric with $\beta > 0$, we obtain for every point $x  = (z,w)\in \RR^n \cong \CC^{j+1} \times \RR^{n-2(j+1)}$ that 
		\begin{align*}
			u(x) = u(\varrho_{\beta}^{2^\beta}(x)) &= u(-\bar{z_2}, \bar{z_1}, z_3, z_4, \dots, z_{j}, z_{j+1}, w)\\
			&= u({z_2}, {z_1}, z_3, z_4, \dots, z_{j}, z_{j+1}, w),
		\end{align*}
		where we have used that $u$ is invariant under rotation in each coordinate of $z$ for this last step. Again, by conjugation with respect to $\rho_j$, the same argument as before (see also Lemma \ref{lem:isCode}) implies that any two adjacent coordinates of $z = (z_1,\dots,z_{j+1})$ may be swapped while preserving the value of $u(z)$. Iterating this principle yields the identity
		\begin{equation}\label{eq:distinctEvenCase}
			u(x) = u(z_{j+1}, z_1, z_2, \dots, z_j, w).
		\end{equation}
		On the other hand, the $(\alpha, \m)$-symmetry of $u$ together with the rotational invariance in each coordinate of $z$ ensures that
		\begin{align*}
			-u(x) = u(\rho_j (x)) &= u(-\overline{z_{j+1}}, \bar{z_1}, \dots, \bar{z_j},w)\\
			&=u(z_{j+1}, z_1, \dots, z_j, w).
		\end{align*}
		Combining this last equation with \eqref{eq:distinctEvenCase}, we conclude that $u(x) = 0$. 
		
		In all cases we have shown that $u \equiv 0$, which is what had to be shown.
	\end{proof}
	
	Next, we aim to verify that the symmetry incompatibilities we have established in Propositions \ref{prop:distinct}-\ref{prop:distinct2} also ensure that the solutions we will produce in Theorem \ref{thm:main} remain distinct even with regards to translations and dilations. This is of important relevance to our problem \eqref{eq:limitProblem} since, if a function $u$ is a solution to \eqref{eq:limitProblem}, then so is the dilated function $x \mapsto \lambda^\gamma u(\lambda x)$, for $\lambda \in \RR$, where $\gamma > 0$ is the homogeneity exponent given by
	\begin{equation}\label{eq:exponent}
		\gamma := \frac{n-bq}{q} = \frac{n-p(1+a)}{p}.
	\end{equation}
	If $a=b=0$, then translation is allowed in the sense that
	\[
	x \mapsto \lambda^\gamma u(\lambda x + y), \quad y\in \RR^n,
	\]
	is also a solution to \eqref{eq:limitProblem}. As a first step in this direction, we show that a non-trivial solution of \eqref{eq:limitProblem} with $(\alpha,\m)$-symmetry cannot be a translation of a $(\beta,\n)$-symmetric solution, provided $(\alpha,\m)$ and $(\beta,\n)$ satisfy one of the conditions dictated by the second part of Theorem \ref{thm:main}.
	
	\begin{lem}\label{lem:noTranslation}
		Fix a dimension $n \ge 4$. Suppose that $(\alpha, \m)$ and $(\beta, \n)$ are two pairs, each of which satisfies condition \eqref{eq:primeCondition}. Let $u,v : \RR^n \to \RR$ be non-trivial solutions to \eqref{eq:limitProblem} which are $(\alpha,\m)$-symmetric and $(\beta,\n)$-symmetric, respectively. Assume that one of the following conditions hold: 
		\begin{enumerate}
			\item $\alpha \ne \beta$;
			\item $\alpha = \beta$ and $\m \ne \n$, with $\m \lesssim \n$ or $\gcd(\m,\n) = 1$.
		\end{enumerate}
		Then, $u$ is not a translation of $v$. 
	\end{lem}
	\begin{proof}
		Suppose for a contradiction that there exists a non-zero $x_0 \in \RR^n$ for which there holds $u(x) = v(x-x_0)$ for all $x \in \RR^n$. We distinguish two cases.
		\begin{enumerate}
			\item If both $(\alpha,\m)$ and $(\beta,\n)$ satisfy \eqref{eq:primeCondition2}, then the transformation 
			\[
			\varsigma : \RR^n \to \RR^n, \quad x \mapsto -x
			\]
			satisfies the rules
			\[
			u(\varsigma (x)) = u(x) \quad \text{and} \quad v(\varsigma (x)) = v(x)
			\]
			for all $x \in \RR^n$. Notice that $\varsigma$ is $\RR$-linear. Consequently, for each $x \in \RR^n$, we have that
			\begin{align*}
				u(-x) = u(x) = v(x-x_0) = v(\varsigma (x-x_0) ) &= v(-x + x_0)\\
				&=u(-x+2x_0).
			\end{align*}
			As $x \in \RR^n$ was arbitrary, we infer that $u$ is periodic and non-trivial; this contradicts $u$ being a non-vanishing solution of \eqref{eq:limitProblem}.
			
			\item Consider next the case where $(\alpha,\m)$ and $(\beta,\n)$ do not both satisfy \eqref{eq:primeCondition2}, and assume without loss of generality that $(\alpha,\m)$ fails this condition. In particular, $n$ must be odd. Then, we consider the transformation
			\[
			\varsigma : \RR^n \to \RR^n, \quad (x_1,\dots, x_{n-1}, x_n)
			\mapsto (-x_1,\dots, -x_{n-1}, x_n).
			\]
			Again, $\varsigma \in G_{\alpha,\m}$ is an $\RR$-linear isometry of $\RR^n$. Furthermore, $u(\varsigma x) = u(x)$ for all $x \in \RR^n$ because $\varsigma$ is composed of sign-preserving components of $G_{\alpha,\m}$. In fact, we can say the same for $v$ because
			\begin{enumerate}[label=(\roman*)]
				\item If $(\beta,\n)$ does not satisfy \eqref{eq:primeCondition2} we get $v(\varsigma x) = v(x)$ by the same reasoning as for $u$.
				\item If $(\beta,\n)$ satisfies \eqref{eq:primeCondition2} there is a non-trivial copy of an orthogonal group acting upon the final coordinates not touched by any sign-changing action from $G_{\beta,\n}$ whence $\varsigma$ is a sign-preserving transformation in $G_{\beta,\n}$ with respect to $\phi_{\beta,\n}$.
			\end{enumerate}
			Now, $\varsigma^2 = I$ and so
			\begin{align*}	
				u(x) = u(\varsigma (x)) = v(\varsigma(x) - x_0) &= v(x - \varsigma(x_0))\\
				&= v(x-\varsigma(x_0) + x_0-x_0)\\
				&= u(x-\varsigma(x_0) + x_0).
			\end{align*}
			Hence, we have periodicity of $u$ provided $\varsigma(x_0) \ne x_0$. Of course, by the same reasoning as earlier, this would be a contradiction. If $\varsigma(x_0) = x_0$, this means that all but the last coordinate of $x_0$ is $0$, at which point we have $\operatorname{Orb}_{G_{\alpha,\m}}(x_0) = \{x_0\}$ by \eqref{eq:primeCondition2}. But then, $v(x) = u(x+x_0)$ is a translation of $u$ by a $G_{\alpha,\m}$-invariant point in space. Thus, $v$ is $\phi_{\alpha,\m}$-equivariant and must be trivial by Propositions \ref{prop:distinct}-\ref{prop:distinct2}. 
		\end{enumerate}
	\end{proof}
	
	The next corollary verifies the distinctness of our solutions under the previously mentioned translations and dilations.
	\begin{cor}
		Fix a dimension $n \ge 4$. Suppose that $(\alpha, \m)$ and $(\beta, \n)$ are two pairs, each of which satisfies condition \eqref{eq:primeCondition}. Let $u,v : \RR^n \to \RR$ be non-trivial solutions to \eqref{eq:limitProblem} which are $(\alpha,\m)$-symmetric and $(\beta,\n)$-symmetric, respectively. Assume one of the following hold true:
		\begin{enumerate}
			\item $\alpha \ne \beta$;
			\item $\alpha = \beta$ and $\m \ne \n$, with $\m \lesssim \n$ or $\gcd(\m,\n) = 1$.
		\end{enumerate}
		Then, $u$ is not a translated rescaling of $v$, in the sense that there does not exist $\lambda \in \RR$ and $y \in \RR^n$ such that
		\[
		u(x) = \lambda^\gamma v(\lambda x+ y).
		\]
		Here, $\gamma > 0$ is the homogeneity exponent given by \eqref{eq:exponent}.
	\end{cor}
	\begin{proof}
		By way of contradiction, suppose we can write
		\[
		u(x) = \lambda^\gamma v(\lambda x + y)
		\]
		for some $\lambda \in \RR$ and $y \in \RR^n$. Since $u$ is non-trivial, we must have $\lambda \ne 0$. If $y = 0$, then $u(x) = w(x)$ where $w(x) := \lambda^\gamma v(\lambda x)$ is itself $\phi_{\beta,\n}$-equivariant. Notice also that $w$ is a solution of \eqref{eq:limitProblem}. Consequently, Propositions \ref{prop:distinct}-\ref{prop:distinct2} imply that $u$ is trivial which contradicts our assumptions. If instead $y \ne 0$, we observe that
		\begin{align*}
			u(x) = \lambda^\gamma v(\lambda x + y) = \lambda^\gamma v\left(\lambda \left[ x + \frac{y}{\lambda}\right]\right)
			&= w\left(x + \frac{y}{\lambda}\right),
		\end{align*}
		meaning that $u$ is a translation of a $\phi_{\beta,\n}$-equivariant solution of \eqref{eq:limitProblem}. Ergo, we are in contradiction with Lemma \ref{lem:noTranslation}.
	\end{proof}

	\section{Existence of Symmetric Palais-Smale Sequences}
	
	In this section, we construct Palais-Smale sequences at suitable energy levels for the extraction of non-trivial solutions to \eqref{eq:limitProblem}. Although our approach here closely mimics that found in Clapp-Rios \cite{clappRios}, we include the details here for the sake of completeness due to the presence of weights and the modified group conditions.

	Recall that, given an open set $\Om \subseteq \RR^n$ containing the origin, we let $\D_a^{1,p}(U,0)$ denote the weighted homogenous Sobolev space obtained by the completion of $C_c^\infty(\Om)$ under the weighted norm 
	\[
	\norm{u}_{\D_a^{1,p}(\Om,0)} := \norm{\nabla u}_{L^p(\Om, \abs{x}^{-ap})} = \left( \int_{\Om} \abs{\nabla u}^p\abs{x}^{-ap}\mathrm{d}x\right)^{1/p}.	
	\]
	As mentioned briefly within the introduction, the Sobolev space $\D_a^{1,p}(\Om,0)$ possesses the expected properties; we refer the reader to \cite[Chapter 2]{mscChernysh} for more detail. 
	
	The energy functional associated to the weighted critical $p$-Laplace problem in \eqref{eq:limitProblem} is the $C^1$-functional $J : \D_a^{1,p}(\RR^n, 0) \to \mathbb{R}$ given by
	\[
	J(u) = \frac{1}{p}\int_{\RR^n} \abs{\nabla u}^p\abs{x}^{-ap}\,\mathrm{d}x - \frac{1}{q}\int_{\RR^n} \abs{u}^q\abs{x}^{-bq}\,\mathrm{d}x\,.
	\]
	A function $u \in \D_a^{1,p}(\RR^n,0)$ is said to be a solution to problem \eqref{eq:limitProblem} if the Fr\'echet derivative of $J$ vanishes on $\D_a^{1,p}(\RR^n,0)$ at $u$, i.e. if
	\[
	\inner{J^\prime(u), h} = 0, \quad \forall h \in \D_a^{1,p}(\RR^n,0)
	\]
	where $\inner{\cdot, \ast}$ denotes the duality pairing and
	\[
	\inner{J^\prime(u), h} = \int_{\RR^n} \abs{\nabla u}^{p-2} \nabla u \cdot \nabla h \abs{x}^{-ap}\d{x} - \int_{\RR^n} \abs{u}^{q-2}uh\abs{x}^{-bq}\d{x}.
	\]

	Recall that, given a Banach space $V$ together with a $C^1$-functional $E$ on $V$, a Palais-Smale sequence (henceforth abbreviated $(PS)$-sequence) for $E$ is a sequence $(v_\alpha)$ in $V$ such that 
	\begin{enumerate}
		\item $E(v_\alpha)$ converges as $\alpha \to \infty$,
		\item $E^\prime(v_\alpha) \to 0$ in the strong operator topology of the dual space $E^\ast$.
	\end{enumerate}
	The limit value $\lim_{\alpha \to \infty} E(v_\alpha)$ is often called the \emph{energy level} of the $(PS)$-sequence.
	
	Throughout this section, unless stated otherwise, $G$ will denote a closed subgroup of $O(n)$, $\Om \subseteq \RR^n$ a $G$-invariant domain containing the origin, and $\phi : G \to \{\pm 1\}$ a continuous homomorphism of groups satisfying (P1).

	We define $\D_a^{1,p}(\Om,0)^\phi$ to be the topological closure of $C_c^\infty(\Om)^\phi$ within the weighted homogeneous Sobolev space $\D_a^{1,p}(\Om,0)$. 	Naturally, as the space $\D_a^{1,p}(\Om,0)$ is a reflexive Banach space for $1 < p < n$, the symmetrized space $\D_a^{1,p}(\Om,0)^\phi$ inherits these properties. We refer the reader to  \cite[Chapter 2]{mscChernysh} and \cite{Chernysh} for a more detailed description of $\D_a^{1,p}(\Om,0)$ and for analogues of well known foundational results from classical Sobolev theory (e.g. standard embedding theorems, Rellich-Kondrachov, and Riesz-Representation results).
	As an intermediary step for Theorem \ref{thm:main}, we shall be forced to consider solutions to a ``symmetrized'' version of problem \eqref{eq:limitProblem} related to the subgroups $G$ and corresponding homomorphisms $\phi$ constructed in the previous section.
	Formally, we shall say that $u \in \D_a^{1,p}(\Om, 0)^\phi$ is a solution to 
	\begin{equation}\label{eq:probOmSym}
		\begin{cases}
			-\operatorname{div}\left(\abs{x}^{-ap}\abs{\nabla u}^{p-2}\nabla u\right) = \abs{x}^{-bq}\abs{u}^{q-2}u &\text{in }\Omega\\
			u\in \D_a^{1,p}(\Om, 0)^\phi
		\end{cases}
	\end{equation}
	provided $\inner{J^\prime(u), h} = 0$ for all $h \in \D_a^{1,p}(\Om, 0)^\phi$. We also point out that, by extension by $0$ outside $\Om$, it is automatic that $\D_a^{1,p}(\Om, 0)^\phi \subseteq \D_a^{1,p}(\RR^n, 0)^\phi \subseteq \D_a^{1,p}(\RR^n,0)$.

	We shall also explicitly introduce an entire analogue of this symmetrized problem in $\RR^n$. Namely, we say that a function $u \in \D_a^{1,p}(\RR^n,0)^\phi$ is a solution to
	\begin{equation}\label{eq:limitProblemSym}
		\begin{cases}
			-\operatorname{div}\left(\abs{x}^{-ap}\abs{\nabla u}^{p-2}\nabla u\right) = \abs{x}^{-bq}\abs{u}^{q-2}u &\text{in }\RR^n\\
			u\in \D_a^{1,p}(\RR^n, 0)^\phi
		\end{cases}
	\end{equation}
	if $\inner{J^\prime(u),h} = 0$ for all $h \in \D_a^{1,p}(\RR^n,0)^\phi$.
	
	\begin{defn}
		We define $\N^\phi(\Om)$ to be the Nehari manifold
		\[	
		\N^\phi(\Om) := \left\{ u \in \D_a^{1,p}(\Om,0)^\phi \setminus \{0\} : \norm{u}^p_{\D_a^{1,p}(\Om,0)} = \norm{u}_{L^q(\Om, \abs{x}^{-bq})}^q \right\}.
		\]
		Subsequently, we set
		\begin{equation}
			c^\phi(\Om) := \inf_{u \in \N^{\phi}(\Om)} J(u) =  \inf_{u \in \N^{\phi}(\Om)} \left(\frac{1}{p} - \frac{1}{q}\right) \int_{\Om} \abs{\nabla u}^p\abs{x}^{-ap}\d{x} \ge 0.
		\end{equation}
	\end{defn}
	By testing a solution $u$ of \eqref{eq:probOmSym} against itself, it is easy to see that the Nehari manifold $\N^\phi(\Om)$ contains all non-trivial solutions of \eqref{eq:probOmSym}, should one exist. We also point out that $\N^\phi(\Om)$ is non-empty; this is verified in this next lemma which compiles several important properties relating to this Nehari manifold and the minimal energy level $c^\phi(\Om)$. We emphasize that the results in this section are standard (see, for instance, Clapp-Rios \cite{clappRios} and Clapp-Vicente-Ben\'itez \cite{clapp2026entire}), but choose to include certain details for the sake of completeness.

	\begin{lem}\label{lem:nehariProperties}
		There hold the following.
		\begin{enumerate}
			\item The Nehari manifold $\N^\phi(\Om)$ is non-empty.
			\item There exists $\mu_0 > 0$ such that $\norm{u}_{\D_a^{1,p}(\Om,0)} \ge \mu_0$ for all $u \in \N^\phi(\Om)$. Consequently, $c^\phi(\Om) > 0$.
			\item We have 
			\[
			c^\phi(\Om) = \inf_{\gamma \in \Gamma} \max_{t \in [0,1]} J(\gamma(t))
			\]
			where $\Gamma \subseteq C\left([0,1], \D_a^{1,p}(\Om,0)^\phi\right)$ consists of those continuous paths $$\gamma : [0,1] \to \D_a^{1,p}(\Om, 0)^\phi$$ beginning at $0$ with non-trivial endpoint of non-positive energy, i.e.
			\[
			\Gamma := \left\{ \gamma \in C\left([0,1], \D_a^{1,p}(\Om,0)^\phi\right) : \gamma(0) = 0,\, \gamma(1) \ne 0,\, J(\gamma(1)) \le 0\right\}.
			\]
			\item $C_c^\infty(\Om)^\phi \cap \N^\phi(\Om)$ is dense in $\N^\phi(\Om)$.
			\item The Nehari manifold $\N^\phi(\Om)$ is a closed $C^1$-Banach submanifold of $\D_a^{1,p}(\Om,0)^\phi$, and a natural constraint for the energy functional $J$.
		\end{enumerate}
	\end{lem}
	\begin{proof}
		\leavevmode
		\begin{enumerate}[leftmargin=*]
			\item  Certainly, since $C_c^\infty(\Om)^\phi$ is non-trivial (it is infinite dimensional) we may choose $u \in C_c^\infty(\Om)^\phi$ such that $u \ne 0$. Consider the mapping
			\[
			\ell : \RR \to \RR, \quad t \mapsto \norm{tu}_{\D_a^{1,p}(\Om,0)}^p - \norm{tu}_{L^q(\Om, \abs{x}^{-bq})}^q
			\]
			which is clearly continuous. By inspection, it is easy to see that $\ell(t) > 0$ for $t$ near $0$ whilst $\ell(t) < 0$ for all $\abs{t}$ sufficiently large. By the Intermediate Value Theorem, it follows that $\ell(t_0 u) = 0$ for some $t_0 \ne 0$ whence $t_0 u \in \N^\phi(\Om)$ since, of course, $t_0 u \ne 0$.
			\item By the CKN inequality, given $u \in \N^\phi(\Om)$, we have
			\begin{align*}
				0&= \norm{u}^p_{\D^{1,p}_a(\Om,0)} - \norm{u}^{q}_{L^q(\Om,\abs{x}^{-bq})}\\
				&\ge \norm{u}^p_{\D^{1,p}_a(\Om,0)} - S\norm{u}^{q}_{\D_a^{1,p}(\Om,0)}\\
				&= \norm{u}^p_{\D^{1,p}_a(\Om,0)} \left(1-S\norm{u}^{q-p}_{\D^{1,p}_a(\Om,0)}\right)
			\end{align*}
			for a suitable constant $S >0$. Since $q > p$, this is only possible if the $\D_a^{1,p}(\Om,0)$-norms of $u \in \N^\phi(\Om)$ cannot be made arbitrarily close to $0$. Thus, there exists $\mu_0 > 0$ such that $\norm{u}_{\D_a^{1,p}(\Om,0)} \ge \mu_0$ for all $u \in \N^\phi(\Om)$, as was asserted.
			
			\item First, given $u \in \N^\phi(\Om)$ we have that $u \ne 0$ and $\norm{u}_{\D^{1,p}_a(\Om,0)}^p = \norm{u}_{L^q(\Om,\abs{x}^{-bq})}^q$ whence, for $t > 1$ sufficiently large, 
			\begin{align*}
				J(tu) &= \frac{\norm{tu}_{\D^{1,p}_a(\Om,0)}^p}{p} - \frac{\norm{tu}_{L^q(\Om,\abs{x}^{-bq})}^q}{q}\\
				&= \left( \frac{t^p}{p} - \frac{t^q}{q}\right) \norm{u}_{\D^{1,p}_a(\Om,0)} \\& < 0.
			\end{align*}
			Furthermore, by the first derivative test, it is easy to see that $J(tu)$ (with $t > 0$) achieves its maximum at $t=1$.
			Let $s_0 > 1$ be such that $J(s_0 u) < 0$ and consider the path $$\gamma_0 : [0,1] \to \D_a^{1,p}(\Om,0)^\phi$$ given by $\gamma_0(t) := ts_0u$. Then, $\gamma_0(0) = 0$, $\gamma_0(1) = s_0 u \ne 0$, and $J(\gamma_0(1)) < 0$. Thus, $\gamma_0 \in \Gamma$. Observing also that
			\[
			\max_{t \in [0,1]} J(\gamma_0(t)) = \max_{t \in [0,1]} J(ts_0u) = J(u), 
			\]
			we infer that
			\begin{align*}
				\inf_{\gamma \in \Gamma} \max_{t \in [0,1]} J(\gamma(t)) &\le \max_{t \in [0,1]} J(\gamma_0(t)) = J(u).
			\end{align*}
			Since $u \in \N^\phi(\Om)$ was arbitrary, it follows that
			\[
			\inf_{\gamma \in \Gamma} \max_{t \in [0,1]} J(\gamma(t))  \le \inf_{u \in N^\phi(\Om)} J(u) = c^\phi(\Om).
			\]
			To establish the reverse inequality, consider the function
			\[
			\kappa : \D_a^{1,p}(\Om,0) \to \RR, \quad u \mapsto \begin{cases} \dfrac{\norm{u}^{q}_{L^q(\Om,\abs{x}^{-bq})}}{\norm{u}^p_{\D^{1,p}_a(\Om,0)}} & \text{if } u \ne 0,\\ 0 & \text{if } u = 0. \end{cases}
			\]	
			By appealing to the CKN inequality, this mapping is easily seen to be continuous: indeed, for all $u \in \D_a^{1,p}(\Om,0)^\phi \setminus \{0\}$ we have
			\begin{align*}
				\abs{\kappa(u)} = \dfrac{\norm{u}^{q}_{L^q(\Om,\abs{x}^{-bq})}}{\norm{u}^p_{\D^{1,p}_a(\Om,0)}} &\le S\norm{u}^{q-p}_{\D^{1,p}_a(\Om,0)}.
			\end{align*}
			Obviously, $\kappa(0) =0$ while, if $u\in \D_a^{1,p}(\Om,0)^\phi$ has non-positive energy (i.e. if $J(u) \le 0$), then
			\begin{align*}
				\kappa(u) = \dfrac{\norm{u}^{q}_{L^q(\Om,\abs{x}^{-bq})}}{\norm{u}^p_{\D^{1,p}_a(\Om,0)}} \ge \frac{q}{p} > 1.
			\end{align*}
			Consequently, for all paths $\gamma \in \Gamma$, we have $\kappa(\gamma(0)) = 0$ and $\kappa(\gamma(1)) > 1$ so that (by the Intermediate Value Theorem) there exists $t_1 \in (0,1)$ with the property that $\kappa(\gamma(t_1)) = 1$. Then, $\gamma(t_1) \in \N^\phi(\Om)$ and 
			\[
			\max_{t \in [0,1]} J(\gamma(t)) \ge J(\gamma(t_1)) \ge c^\phi(\Om).
			\]
			As this holds for all paths $\gamma \in \Gamma$, we deduce that
			\[
			\inf_{\gamma \in \Gamma} \max_{t \in [0,1]} J(\gamma(t)) \ge c^\phi(\Om),
			\]
			whence 
			\[
			\inf_{\gamma \in \Gamma} \max_{t \in [0,1]} J(\gamma(t)) = c^\phi(\Om),
			\]
			as was asserted.
			\item For this point, let $u \in \N^\phi(\Om)$ be given. Then, $u \ne 0$, $u \in \D^{1,p}_a(\Om,0)^\phi$, and $\norm{u}_{\D^{1,p}_a(\Om,0)}^p =\norm{u}_{L^q(\Om,\abs{x}^{-bq})}^p$. Since $C_c^\infty(\Om)^\phi$ is dense in $\D_a^{1,p}(\Om,0)^\phi$ by definition, we may select a sequence $(u_\alpha)$ in $C_c^\infty(\Om)^\phi$ such that $u_\alpha \to u$ in $\D_a^{1,p}(\Om,0)$, as $ \alpha \to \infty$. Without loss of generality, we may assume that each $u_\alpha\ne 0$. For each index $\alpha \in \NN$, we define
			\[
			t_\alpha := \left(\frac{\norm{u_\alpha}_{\D_a^{1,p}(\Om,0)}^p}{ \norm{u_\alpha}^q_{L^q(\Om,\abs{x}^{-bq})} } \right)^{\frac{1}{q-p}} > 0
			\]
			and set $v_\alpha := t_\alpha u_\alpha$. Then, $v_\alpha \in C_c^\infty(\Om)^\phi \subset \D^{1,p}_a(\Om,0)^\phi$ for each $\alpha \in\NN$ and $v_\alpha \in \N^\phi(\Om)$ since
			\begingroup
			\small
			\begin{align*}
				\norm{v_\alpha}_{\D_a^{1,p}(\Om,0)}^p&= \norm{t_\alpha u_\alpha}_{\D_a^{1,p}(\Om,0)}^p \\
				&= t_\alpha^p \norm{ u_\alpha}_{\D_a^{1,p}(\Om,0)}^p\\
				&= \left[\frac{\norm{u_\alpha}_{\D_a^{1,p}(\Om,0)}^p}{ \norm{u_\alpha}^q_{L^q(\Om,\abs{x}^{-bq})} } \right]^{\frac{p}{q-p}}  \norm{ u_\alpha}_{\D_a^{1,p}(\Om,0)}^p\\
				&= \left[\frac{\norm{u_\alpha}_{\D_a^{1,p}(\Om,0)}^p}{ \norm{u_\alpha}^q_{L^q(\Om,\abs{x}^{-bq})} } \right]^{\frac{q}{q-p}} \left[\frac{\norm{u_\alpha}_{\D_a^{1,p}(\Om,0)}^p}{ \norm{u_\alpha}^q_{L^q(\Om,\abs{x}^{-bq})} } \right]^{\frac{p-q}{q-p}}   \norm{ u_\alpha}_{\D_a^{1,p}(\Om,0)}^p\\
				&= t_\alpha^q  \left[\frac{ \norm{u_\alpha}^q_{L^q(\Om,\abs{x}^{-bq})}}{\norm{u_\alpha}_{\D_a^{1,p}(\Om,0)}^p} \right]^{\frac{q-p}{q-p}}   \norm{ u_\alpha}_{\D_a^{1,p}(\Om,0)}^p\\
				&= t_\alpha^q \norm{u_\alpha}^q_{L^q(\Om,\abs{x}^{-bq})}\\
				&= \norm{v_\alpha}^q_{L^q(\Om,\abs{x}^{-bq})}.
			\end{align*}
			\endgroup
			In addition, by continuity of the $\D_a^{1,p}(\Om, 0)$-norm and the CKN inequality, it follows from $u_\alpha \to u$ in $\D_a^{1,p}(\Om,0)$ that $t_\alpha \to 1$, as $\alpha \to \infty$. Consequently,
			\begin{align*}
				\norm{v_\alpha - u}_{\D_a^{1,p}(\Om,0)} &\le \norm{v_\alpha - u_\alpha}_{\D_a^{1,p}(\Om,0)} + \norm{u_\alpha - u}_{\D_a^{1,p}(\Om,0)}\\
				&= \norm{t_\alpha u_\alpha - u_\alpha}_{\D_a^{1,p}(\Om,0)} + o(1)\\
				&= (t_\alpha-1)\norm{u_\alpha}_{\D_a^{1,p}(\Om,0)} + o(1)\\
				&= o(1)
			\end{align*}
			whence we infer that $v_\alpha \to u$ strongly in $\D_a^{1,p}(\Om,0)$. 
			
			\item This follows by a verbatim extension of the argument used within Clapp-Rios \cite[Lemma 2.3]{clappRios}. \qedhere
		\end{enumerate}
	\end{proof}
	
	As an important point, we now verify that this minimal energy level $c^\phi(\Om)$ is independent of our choice of $G$-invariant domain $\Om$. It is precisely this energy invariance that will ultimately allow us to conclude that our extracted function indeed solves the entire problem posed in \eqref{eq:limitProblem}.
	
	\begin{lem}\label{lem:minimalEnergyInvariance}
		We have $c^\phi(\Om) = c^\phi$, where we define $c^\phi := c^\phi(\RR^n)$. 
	\end{lem}
	\begin{proof}
		Because $\D_a^{1,p}(\Om,0)^\phi \subseteq \D_a^{1,p}(\RR^n,0)^\phi$ after an extension by 0 outside $\Om$, it is automatic that $c^\phi \le c^\phi(\Om)$. In order to establish the reverse inequality, let $u \in \N^\phi(\RR^n)$ and select a sequence $(u_\alpha)$ in $C_c^\infty(\RR^n)^\phi \cap \N^\phi(\RR^n)$ converging to $u$ in $\D_a^{1,p}(\RR^n,0)$ as $\alpha \to \infty$. Since each $u_\alpha$ has compact support and $\Om$ contains the origin, we may select for each index $\alpha \in \NN$ some $\lambda_\alpha > 0$ such that the rescaled function
		\[
		v_\alpha(x) := \lambda_\alpha^{-\gamma} u_\alpha\left(\frac{x}{\lambda_\alpha}\right)
		\]
		is compactly supported in $\Om$. Here, $\gamma > 0$ is the homogeneity exponent, associated to the CKN-inequality, defined in \eqref{eq:exponent}. By this homogeneity property, $\norm{v_\alpha}_{\D_a^{1,p}(\RR^n,0)} = \norm{u_\alpha}_{\D_a^{1,p}(\RR^n,0)}$ and $\norm{u_\alpha}_{L^q(\RR^n,0)} = \norm{v_\alpha}_{L^q(\RR^n,0)}$ which implies that $J(v_\alpha) = J(u_\alpha)$ and $v_\alpha \in \N^\phi(\Om)$. Finally, observe that
		\[
		c^\phi(\Om) \le J(v_\alpha), \quad \forall \alpha \in \NN
		\]
		whence $c^\phi(\Om) \le \lim_{\alpha \to \infty} J(v_\alpha) = \lim_{\alpha \to \infty} J(u_\alpha) = J(u)$. As $u \in \N^\phi(\RR^n)$ is arbitrary, it follows that $c^\phi(\Om) \le c^\phi$.
	\end{proof}

	\begin{cor}[Existence of $(PS)$-Sequence at the minimal energy level $c^\phi$]\label{cor:PSExistence}
		There exists a $(PS)$-sequence in $\D_a^{1,p}(\Om,0)^\phi$ at energy level $c^\phi$. Namely, there exists a sequence $(u_\alpha)$ in $\D_a^{1,p}(\Om,0)^\phi$ such that
		\[
		\lim_{\alpha \to \infty} J(u_\alpha) = c^\phi \quad \text{and} \quad J^\prime(u_\alpha) \to 0 \text{ in } \D_a^{-1,p^\prime}(\Om,0)^\phi. 
		\]
	\end{cor}
	\begin{proof}
		By Lemma \ref{lem:minimalEnergyInvariance}, it suffices to construct such a $(PS)$-sequence at the energy level $c^\phi(\Om)$.	Letting $\Gamma$ be as in Lemma \ref{lem:nehariProperties}, notice that given any path $\gamma \in \Gamma$, the composition $t \mapsto J(\gamma(t))$ is a continuous mapping $[0,1] \to \RR$ and is therefore bounded. Consequently,
		\[	
		c^\phi(\Om) = \inf_{\gamma \in \Gamma} \max_{t \in [0,1]} J(\gamma(t)) < \infty.
		\]
		Next, there exists by Lemma \ref{lem:nehariProperties} a constant $\mu_0 > 0$ such that for all $u \in \N^\phi(\Om)$ there holds $\norm{u}_{\D_a^{1,p}(\Om,0)} \ge \mu_0$. It follows that
		\begin{align*}
			J(u) &= \frac{\norm{u}_{\D_a^{1,p}(\Om,0)} }{p} - \frac{\norm{u}_{L^q(\Om,\abs{x}^{-bq})} }{q}\\
			&= \norm{u}^p_{\D^{1,p}_a(\Om,0)} \left( \frac{1}{p}- \frac{1}{q}\right)\\
			&\ge \mu_0^p\left( \frac{1}{p}- \frac{1}{q}\right)\\
			& > 0
		\end{align*}
		uniformly over $u \in \N^\phi(\Om)$. Thus, 
		\[
		\inf_{\gamma \in \Gamma} \max_{t \in [0,1]} J(\gamma(t))= c^\phi(\Om) \ge \mu_0^p\left( \frac{1}{p}- \frac{1}{q}\right) > 0
		\]
		whilst $J(\gamma(0)) = 0$ and $J(\gamma(1)) \le 0$ for all $\gamma \in \Gamma$. Consequently, the assertion follows after an application of Theorems 2.8-2.9 from Willem \cite{willem2012minimax}.
	\end{proof}

	\section{Palais-Smale Symmetric Criticality}
	
	For this section, $G \subseteq O(n)$ is a closed subgroup and $\phi : G\to\{\pm 1\}$ a continuous group homomorphism such that (P1) holds. Moreover, we let $\Om \subseteq \mathbb{R}^n$ be a $G$-invariant domain containing the origin. The next proposition asserts that Palais-Smale sequences in the symmetric subspace $\D_a^{1,p}(\Om,0)^\phi$ are also Palais-Smale within the global space $\D_a^{1,p}(\Om, 0)$. The proof of this proposition utilizes a calculation from Clapp-Rios \cite[Lemma A.1]{clappRios}, together with a straightforward application of Minkowski's inequality for integrals to make the duality connection. We note that this observation allows for a more direct application of the known Struwe-type decomposition for weighted equations of this type.

	\begin{prop}\label{prop:PSSymmetry}
		Suppose $(u_\alpha) \subset \D^{1,p}_a(\Om,0)^\phi$ is a Palais-Smale sequence for $J$ in $\D_a^{1,p}(\Om, 0)^\phi$. That is, as $\alpha \to \infty$, $J(u_\alpha)$ converges to some real number and
		\[
		J^\prime(u_\alpha) \to 0 \text{ in } \D^{-1,p^\prime}_a(\Om,0)^\phi.
		\]
		Then, $(u_\alpha)$ is a Palais-Smale sequence for $J$ in the non-symmetrized space $\D_a^{1,p}(\Om, 0)$. That is,
		\[
		J^\prime(u_\alpha) \to 0 \text{ in } \D^{-1,p^\prime}_a(\Om,0).
		\]
		Moreover, if $u \in \D_a^{1,p}(\Om,0)^\phi$ solves \eqref{eq:probOmSym}, then $\inner{J^\prime(u), h} = 0$ for all $h \in \D_a^{1,p}(\Om,0)$.
	\end{prop}
	\begin{proof}
		Let $h \in C_c^\infty(\Om)$ be an arbitrary test function. Let $\mu$ denote the normalized Haar measure on $G$, so that $\mu(G) = 1$, and define
		$$
		\vartheta(x) = \int_G\phi(g)h(gx)\,\d\mu.
		$$
		This rule clearly defines a function $\vartheta \in C_c^\infty(\Om)^\phi$. Moreover, by Minkowski's inequality for integrals 
		\[
		\begin{aligned}
			\norm{\vartheta}_{\D_a^{1,p}(\Om, 0)}
			&=\left(\int_\Om \abs{\int_G\phi(g) g^{-1}\nabla h(gx)\,\d\mu}^p \abs{x}^{-ap}\d{x}\right)^{1/p}\\
			&\le \int_G \left(\int_\Om \abs{\phi(g) g^{-1}\nabla h(gx)}^p\abs{x}^{-ap}\d{x}\right)^{1/p}\d\mu.
		\end{aligned}
		\]
		Then, simplifying, using that $g, g^{-1}$ are linear isometries, and performing a change of variables, we obtain
		\begin{align}
			\norm{\vartheta}_{\D_a^{1,p}(\Om, 0)}
			&\le\int_G\left(\int_\Om \abs{\nabla h(gx)}^p\abs{gx}^{-ap}\d{x}\right)^{1/p}\d\mu\notag\\
			&=\int_G\left(\int_\Om \abs{\nabla h(y)}^p\abs{y}^{-ap}\d{y}\right)^{1/p}\d\mu\notag\\
			&=\norm{h}_{\D_a^{1,p}(\Om, 0)}\label{eq:sym_norm_comp}
		\end{align}
		Next, as in Clapp-Rios \cite[Lemma A.1]{clappRios}, if $u\in \D_a^{1,p}(\Om , 0)^\phi$ then $u(gx) = \phi(g)u(x)$ and $\phi(g)\nabla u(x) = g^{-1}\nabla u (gx)$ so
		\[
		\begin{aligned}
			&\inner{J^\prime(u), \vartheta}\\
			&= \int_{\Om}\abs{\nabla u(x)}^{p-2}\nabla u(x) \cdot\nabla \vartheta(x)\abs{x}^{-ap}\d{x} - \int_\Om \abs{u(x)}^{q-2}u(x)\vartheta(x)\abs{x}^{-bq}\d{x}\\
			&= \int_{\Om} \int_G\abs{\nabla u(x)}^{p-2}\phi(g)\nabla u(x) \cdot (g^{-1}\nabla h(gx)) \abs{x}^{-ap} \d\mu\d{x} \\
			&\quad - \int_\Om \int_G\abs{u(x)}^{q-2}\phi(g)u(x)h(gx)\abs{x}^{-bq}\d\mu \d{x}\\
			&= \int_{\Om} \int_G\abs{\nabla u(gx)}^{p-2}(g^{-1}\nabla u(gx)) \cdot (g^{-1}\nabla h(gx))\abs{x}^{-ap}\d\mu \d{x} \\
			&\quad - \int_\Om \int_G\abs{u(x)}^{q-2}u(gx)h(gx)\abs{x}^{-bq}\d\mu \d{x}.
		\end{aligned}
		\]
		Then, using that $g^{-1}$ is a linear isometry and applying Fubini's Theorem, we get
		\[
		\begin{aligned}
			&\inner{J^\prime(u), \vartheta}\\
			&= \int_G\int_{\Om} \abs{\nabla u(gx)}^{p-2}\nabla u(gx) \cdot \nabla h(gx) \abs{gx}^{-ap}\d{x}\d\mu \\
			&\quad - \int_G\int_{\Om} \abs{u(gx)}^{q-2}u(gx)h(gx) \abs{gx}^{-bq}\d{x}\d\mu\\
			&= \int_G\int_{\Om} \abs{\nabla u(y)}^{p-2}\nabla u(y) \cdot \nabla h(y) \abs{y}^{-ap}\d{y}\d\mu \\
			&\quad - \int_G\int_{\Om} \abs{u(y)}^{q-2}u(y)h(y) \abs{y}^{-bq}\d{y}\d\mu\\
			&=\inner{J^\prime(u), h}.
		\end{aligned}
		\]
		In particular,
		\[
		\inner{J^\prime(u_\alpha), h} = \inner{J^\prime(u_\alpha), \vartheta}.
		\]
		Consequently, if $u$ solves \eqref{eq:probOmSym}, then $\inner{J^\prime(u),h} = 0$ for all $h \in \D_a^{1,p}(\Om,0)$.
		Finally, combining this with \eqref{eq:sym_norm_comp} and recalling that $\vartheta \in C_c^\infty(\Om)^\phi \subseteq \D_a^{1,p}(\Om ,0)^\phi$, we see that
		\[
		\begin{aligned}
			\abs{\inner{J^\prime(u_\alpha), h}} 
			= \abs{\inner{J^\prime(u_\alpha), \vartheta}} 
			&\le \norm{J^\prime(u_\alpha)}_{\D_a^{-1, p^\prime}(\Om, 0)^\phi}\norm{\vartheta}_{\D_a^{1,p}(\Om,0)}\\
			&\le \norm{J^\prime(u_\alpha)}_{\D_a^{-1, p^\prime}(\Om, 0)^\phi}\norm{h}_{\D_a^{1,p}(\Om, 0)}.
		\end{aligned}
		\]
		The assertion now follows by a density argument. 
	\end{proof}

	\section{Symmetric Global Compactness}
	
	We now establish a symmetric global compactness result for Palais-Smale sequences in the symmetrized space. Throughout this section, \(G\) denotes a closed subgroup of the orthogonal group \(O(n)\) acting on a $G$-invariant bounded domain \(\Omega \subset \mathbb{R}^n\), containing the origin, and \(\phi : G \to \{\pm 1\}\) is a continuous homomorphism of groups. Furthermore, $(G,\phi)$ are assumed to satisfy (P1)-(P2).

	\begin{prop}\label{prop:struwea<b}
		Assume $a < b$.
		Suppose there exists a Palais-Smale sequence $(u_\alpha) \subset \D^{1,p}_a(\Om,0)^\phi$ for $J$ such that, as $\alpha \to \infty$,
		\[
		J(u_\alpha) \to c^\phi, \quad J^\prime(u_\alpha) \to 0 \text{ in } \D^{-1,p^\prime}_a(\Om,0)^\phi.
		\]
		Then, there exists a non-trivial solution to \eqref{eq:limitProblemSym}.
	\end{prop}
	
	\begin{proof}
		By an application of Proposition \ref{prop:PSSymmetry}, the sequence $(u_\alpha)$ is Palais-Smale in $\D_a^{1,p}(\Om,0)$ with energy level $c^\phi$. Consequently, by \cite[Lemma 6.1]{Chernysh}, there exists $v_0 \in \D_a^{1,p} (\Om,0)$ such that after passing to a subsequence, there holds $u_\alpha \rightharpoonup v_0$ in $\D_a^{1,p} (\Om,0)$,
		\begin{enumerate}
			\item $\nabla u_\alpha \to \nabla v_0$ pointwise a.e. on $\Om$,
			\item $J^\prime(v_0) = 0$ on $\D_a^{1,p}(\Om,0)$,
			\item $J(u_\alpha - v_0) \to c^\phi - J(v_0)$,
			\item $J^\prime(u_\alpha - v_0) \to 0$ in $\D_a^{-1,p^\prime}(\Om,0)$.
		\end{enumerate}
		In fact, since $\D_a^{1,p} (\Om,0)^\phi$ is a (weakly) closed subspace of $\D_a^{1,p} (\Om,0)$, we must have $v_0 \in \D_a^{1,p} (\Om,0)^\phi$. Furthermore, it may assumed that $u_\alpha \to v_0$ pointwise a.e. on $\Om$. 
		
		We now distinguish two possible cases. First, if $v_0 \ne 0$, then $v_0 \in \N^\phi(\Om)$ so $J(v_0) \ge c^\phi$ whence $c^\phi - J(v_0) \le 0$. However, being a Palais-Smale sequence for $J$, we see that $(u_\alpha - v_0)$ has non-negative limiting energy, i.e. by \cite[Lemma 2.4]{Chernysh}
		\[
		\liminf_{\alpha \to \infty} J(u_\alpha - v_0) \ge 0.
		\]
		Hence, we infer that $c^\phi - J(v_0) \ge 0$ as well yielding $J(v_0) = c^\phi$. Thus, after an extension by $0$, we have $v_0 \in \N^\phi(\RR^n)$ and so $v_0$ solves \eqref{eq:limitProblemSym}. 
		
		It remains to handle the case where $v_0 = 0$. In this case, proceeding as in \cite[Theorem 1]{Chernysh}, we find sequences $(y_\alpha)$ in $\overline{\Om}$ and $(\lambda_\alpha)$ in $(0, \infty)$ such that
		\[
		\delta = \sup_{y\in\overline{\Om}}\int_{B(y, \lambda_\alpha)}\abs{u_\alpha}^q\abs{x}^{-bq}\d{x} = \int_{B(y_\alpha, \lambda_\alpha)}\abs{u_\alpha}^q\abs{x}^{-bq}\d{x}
		\]
		where $\delta > 0$ is chosen sufficiently small. Arguing as in Step 4 of \cite[Theorem 1]{Chernysh}, because $a < b$, we can show that (up to a subsequence)
		\[
		\frac{y_\alpha}{\lambda_\alpha} \to x_0
		\]
		for some $x_0\in\RR^n$. We may also assume that $\abs{x_0 - \frac{y_\alpha}{\lambda_\alpha}} < 1$ for all indices $\alpha \in \NN$. Then, setting
		\[
		v_\alpha(x) := \lambda_\alpha^{\gamma}u_\alpha(\lambda_\alpha x)\,,
		\]
		we obtain a bounded sequence in $\D_a^{1,p}(\RR^n,0)^\phi$ such that, for all $\alpha \in \NN$,
		\[
		\int_{B(x_0, 2)}\abs{v_\alpha}^q\abs{x}^{-bq}\d{x} \ge \delta.
		\] 
		Up to a subsequence, we may suppose $(v_\alpha)$ converges weakly to some $v \in \D_a^{1,p}(\RR^n, 0)$. By the inequality above, it follows that $v$ is non-trivial by a duality argument (see, for instance, the proof of Step 6 in \cite[Theorem 1]{Chernysh} or Step 3 in Mercuri-Willem \cite[Theorem 1.2]{mercuri-willem}). Moreover, arguing as in Step 7 in \cite[Theorem 1]{Chernysh}, we may suppose that $\lambda_\alpha \to 0$.

		Notice that $v$ is indeed an element of $\D_a^{1,p}(\RR^n,0)^\phi$ since  $\D_a^{1,p}(\RR^n,0)^\phi$ is a closed subspace of  $\D_a^{1,p}(\RR^n,0)$ and $v$ is, by definition, the weak limit of a sequence in $\D_a^{1,p}(\RR^n,0)^\phi$. Proceeding now as in \cite[Proposition 6.2]{Chernysh} we see that $v \in \D_a^{1,p}(\RR^n,0)^\phi$ is a solution to problem \eqref{eq:limitProblem}, and hence to \eqref{eq:limitProblemSym}. Furthermore, being that $v$ is non-trivial, $J(v) \ge c^\phi$ and, defining
		\[
		w_\alpha(z) := u_\alpha(z) - \lambda_\alpha^{-\gamma}v\left(\frac{z}{\lambda_\alpha}\right)\,,
		\]
		we have
		\begin{enumerate}
			\item $J(w_\alpha ) \to c^\phi - J(v)$;
			\item $J^\prime(w_\alpha) \to 0$ strongly in the dual of $\D_a^{1,p}(\Om,0)$.
		\end{enumerate}
		As before, $\liminf J(w_\alpha) \ge 0$ and, yet, $c^\phi - J(v) \le 0$. It thereby follows that $J(v) = c^\phi$ and so $v$ is a non-trivial solution to \eqref{eq:limitProblemSym}.

	\end{proof}
	
	\begin{prop}\label{prop:struwea=b}
		Assume $a = b \ne 0$ and $(G, \phi)$ is such that (P3) holds true. Suppose there exists a Palais-Smale sequence $(u_\alpha) \subset \D^{1,p}_a(\Om,0)^\phi$ for $J$ such that, as $\alpha \to \infty$,
		\[
		J(u_\alpha) \to c^\phi, \quad J^\prime(u_\alpha) \to 0 \text{ in } \D^{-1,p^\prime}_a(\Om,0)^\phi.
		\]
		Then, there exists non-trivial solution to \eqref{eq:limitProblemSym}.
	\end{prop}
	
	\begin{proof} As before, Proposition \ref{prop:PSSymmetry} implies that the given (PS)-sequence is also Palais-Smale when considered within the larger space $\D_a^{1,p}(\Om,0)$. Hence, after passing to a subsequence, there exists a weak limit $v_0 \in \D_a^{1,p} (\Om,0)$ such that $u_\alpha \rightharpoonup v_0$ in $\D_a^{1,p} (\Om,0)$ and pointwise a.e. on $\Om$, as $\alpha \to \infty$. Moreover,
		\begin{enumerate}
			\item $\nabla u \to \nabla v_0$ pointwise a.e. on $\Om$,
			\item $J^\prime(v_0) = 0$ on $\D_a^{1,p}(\Om, 0)$,
			\item $\phi(u_\alpha - v_0) \to c^\phi - \phi(v_0)$,
			\item $\phi^\prime(u_\alpha - v_0) \to 0$ in $\D_a^{-1,p^\prime}(\Om,0)$.
		\end{enumerate}
		Of course, being the weak limit of a sequence in $\D_a^{1,p}(\Om,0)^\phi$, we must have that $v_0$ is $\phi$-equivariant, i.e. $v_0 \in \D_a^{1,p}(\Om,0)^\phi$. 
		We now distinguish two possible cases. First, if $v_0 \ne 0$, then we proceed exactly as in the first part of Proposition \ref{prop:struwea<b}. If instead $v_0 = 0$, we once again select sequences $(y_\alpha)$ in $\overline{\Om}$ and $(\lambda_\alpha)$ in $(0, \infty)$ such that
		\[
		\delta = \sup_{y\in\overline{\Om}}\int_{B(y, \lambda_\alpha)}\abs{u_\alpha}^q\abs{x}^{-bq}\d{x} = \int_{B(y_\alpha, \lambda_\alpha)}\abs{u_\alpha}^q\abs{x}^{-bq}\d{x}
		\]
		where $\delta > 0$ is chosen sufficiently small. Then, by leveraging (P3), we can show that (up to a subsequence)
		\begin{equation}\label{eq:psPropeq1}
			\frac{y_\alpha}{\lambda_\alpha} \to x_0
		\end{equation}
		for some $x_0\in\RR^n$. Indeed, should this sequence be unbounded, we may pass to a subsequence in such a way that
		\[
		\frac{\abs{y_\alpha}}{\lambda_\alpha} \to \infty, \quad \abs{y_\alpha} > 4\lambda_\alpha.
		\]
		Then, each $y_\alpha \ne 0$ so $y_\alpha/\abs{y_\alpha}$ is a sequence on the unit sphere. Passing to yet another subsequence, we may assume that
		\[
		\frac{y_\alpha}{\abs{y_\alpha}} \to y, \quad \text{as $\alpha \to \infty$}.
		\]
		Appealing to condition (P3) and borrowing ideas from Clapp-Rios \cite{clappRios},
		given any $m \in \NN$, we may select elements $g_1,\dots,g_m \in G$ such that $g_i y \ne g_j y$, for all $i \ne j$. Consequently, there exists $\epsilon > 0$ so that
		\[
		\abs{g_i y - g_j y} \ge 2\epsilon > 0, \quad \forall i \ne j.
		\]
		Using then that $\frac{y_\alpha}{\abs{y_\alpha}} \to y$, it follows from this last identity that
		\begin{equation}
			\abs{g_i \frac{y_\alpha}{\abs{y_\alpha}} -g_j \frac{y_\alpha}{\abs{y_\alpha}}}  \ge \epsilon
		\end{equation}
		for all $\alpha \in \NN$ sufficiently large. Hence, for all such $\alpha$,
		\begin{equation}
			\frac{\abs{ g_i y_\alpha - g_j y_\alpha }}{\lambda_\alpha} \ge \epsilon \frac{\abs{y_\alpha}}{\lambda_\alpha}.
		\end{equation}
		Because the right hand side tends to $+\infty$, we infer that $$B(g_i y_\alpha, \lambda_\alpha) \cap B(g_j y_\alpha, \lambda_\alpha)=\varnothing,$$ for all $i \ne j$ and $\alpha$ large. But,
		\begin{align*}
			\delta &= \int_{B(y_\alpha,\lambda_\alpha)} \abs{u_\alpha}^q\abs{x}^{-bq}\d{x}\\
			&=  \int_{B(g_i y_\alpha,\lambda_\alpha)} \abs{u_\alpha(g_i^{-1} z)}^q\abs{g_i^{-1} z}^{-bq}\d{z}, && (z := g_i x)\\
			&= \int_{B(g_i y_\alpha,\lambda_\alpha)} \abs{\phi(g_i^{-1})u_\alpha(z)}^q\abs{z}^{-bq}\d{z}\\
			&=	\int_{B(g_i y_\alpha,\lambda_\alpha)} \abs{u_\alpha(z)}^q\abs{z}^{-bq}\d{z}
		\end{align*}
		It thereby follows that, for all $\alpha$ sufficiently large,
		\begin{align*}
			\int_{\RR^n} \abs{u_\alpha}^q\abs{x}^{-bq}\d{x} \ge \sum_{i=1}^m \int_{B(g_i y_\alpha,\lambda_\alpha)} \abs{u_\alpha(z)}^q\abs{z}^{-bq}\d{z} = m\delta,
		\end{align*}
		which contradicts the fact that $(u_\alpha)$ is a bounded sequence in $\D_a^{1,p}(\RR^n,0)$. In short: we have proven that $y_\alpha/\lambda_\alpha$ is bounded and thus, passing to a subsequence, we may assume \eqref{eq:psPropeq1}. From here, we complete the argument by following the second part of the proof from Proposition \ref{prop:struwea<b}.
	\end{proof}

	We are now ready to supply the proof of our main result. 
	
	\begin{proof}[Proof of Theorem \ref{thm:main}]
		Let $n \ge 4$ and put $k = k(n) := \left\lfloor \frac{n}{2}\right\rfloor - 1$. For each $\alpha \ge 0$ and every tuple $\m = (m_1,\dots, m_k)$ such that $(\alpha,\m)$ satisfies condition \eqref{eq:primeCondition}, the construction detailed in \S2 yields a group and homomorphism pair $(G_{\alpha,\m}, \phi_{\alpha,\m})$ satisfying properties (P1)-(P2) by virtue of Lemma \ref{lem:primeConstruct}. 
		
		If $a < b$ in the problem parameters of \eqref{eq:limitProblem}, an application of Corollary \ref{cor:PSExistence} together with Proposition \ref{prop:struwea<b}, for $\Om$ equal to the unit ball in $\RR^n$, produces, for any such pair $(\alpha,\m)$, a non-trivial solution $u_{\alpha,\m}$ to problem \eqref{eq:limitProblemSym}  which is $(\alpha,\m)$-symmetric. Since $\phi_{\alpha,\m}$ is surjective and $u_{\alpha,\m}$ is non-trivial, it is automatic that $u_{\alpha,\m}$ is sign-changing. Then, Proposition \ref{prop:PSSymmetry} applied with $\Om = \RR^n$ ensures that $u_{\alpha,\m}$ solves the unrestricted problem \eqref{eq:limitProblem}. 
		
		If we are instead in the case $a=b \ne 0$, and the pair $(\alpha,\m)$ satisfies the additional condition \eqref{eq:primeCondition2}, the argument proceeds exactly as in the case $a < b$, but replaces Proposition \ref{prop:struwea<b} with Proposition \ref{prop:struwea=b}.
		
		Finally, we treat the case $a=b=0$. Although condition (P3) will not hold unless \eqref{eq:primeCondition2} is satisfied, it is easy to see from the definition of $G_{\alpha,\m}$ and \eqref{eq:orbitForm} that the weaker condition
		\[
		\dim \left( \operatorname{Orb}_{G_{\alpha,\m}}(x)\right) > 0 \quad \text{or}
		\quad \operatorname{Orb}_{G_{\alpha,\m}}(x) = \{x\}
		\]
		is valid for all $x \in \RR^n$. Consequently, an application of Clapp-Rios \cite[Theorem 3.1]{clappRios} with $\Om$ equal to the unit ball ensures the existence of a non-trivial solution $u_{\alpha,\m}$ of the problem \eqref{eq:unweighted} that is $(\alpha,\m)$-symmetric. Furthermore, being that $\phi_{\alpha,\m}$ is surjective by construction, each $u_{\alpha,\m}$ is sign-changing. 
		
		The second part of the theorem follows by combining Propositions \ref{prop:distinct}-\ref{prop:distinct2}.
		
	\end{proof}
	
	\bibliographystyle{plain}
	\bibliography{mybib}

\begin{thebibliography}{10}

\bibitem{bracho2005symmetry}
Javier Bracho, M{\'o}nica Clapp, and Wac{\l}aw Marzantowicz.
\newblock Symmetry breaking solutions of nonlinear elliptic systems.
\newblock {\em Toplogical Methods in Nonlinear Analysis, Journal of the Juliusz
  Schauder Center}, 26:189--201, 2005.

\bibitem{CaffarelliGidasSpruck}
Luis Caffarelli, Basilis Gidas, and Joel Spruck.
\newblock Asymptotic symmetry and local behavior of semilinear elliptic
  equations with critical sobolev growth.
\newblock {\em Communications on Pure and Applied Mathematics}, 42:271--297,
  1989.

\bibitem{caffarelli1984first}
Luis Caffarelli, Robert Kohn, and Louis Nirenberg.
\newblock First order interpolation inequalities with weights.
\newblock {\em Compositio Mathematica}, 53(3):259--275, 1984.

\bibitem{catinoMonticelliRoncorini}
Giovanni Catino, Dario~D. Monticelli, and Alberto Roncoroni.
\newblock On the critical {$p$}-{L}aplace equation.
\newblock {\em Advances in Mathematics}, 433:109331, 2023.

\bibitem{mscChernysh}
Edward Chernysh.
\newblock A global compactness theorem for critical $p$-laplace equations with
  weights.
\newblock Master's thesis, McGill University, Montreal, Quebec, April 2020.
\newblock Department of Mathematics and Statistics.

\bibitem{Chernysh}
Edward Chernysh.
\newblock A {S}truwe-type decomposition result for weighted critical
  $p$-{L}aplace equations.
\newblock {\em arXiv preprint}, arXiv:2410.14861, 2024.

\bibitem{ciraolo-corso}
Giulio Ciraolo and Rosario Corso.
\newblock Symmetry for positive critical points of
  caffarelli–kohn–nirenberg inequalities.
\newblock {\em Nonlinear Analysis}, 216:112683, 2022.

\bibitem{ciraolo2025classificationresultsboundedpositive}
Giulio Ciraolo and Michele Gatti.
\newblock Classification results for bounded positive solutions to the critical
  $p$-{L}aplace equation, 2025.
\newblock arXiv:2510.23243.

\bibitem{ciraoloPolvara}
Giulio Ciraolo and Camilla~Chiara Polvara.
\newblock On the classification of extremals of caffarelli-kohn-nirenberg
  inequalities.
\newblock {\em Calculus of Variations and Partial Differential Equations},
  64(8):246, 2025.

\bibitem{Clapp_pure}
M{\'o}nica Clapp.
\newblock Entire nodal solutions to the pure critical exponent problem arising
  from concentration.
\newblock {\em Journal of Differential Equations}, 261(6):3042--3060, 2016.

\bibitem{clapp2025optimal}
M{\'o}nica Clapp, Alberto Salda{\~n}a, Mayra Soares, and Victor~A
  Vicente-Ben\'itez.
\newblock Optimal pinwheel partitions and pinwheel solutions to a nonlinear
  {S}chr{\"o}dinger system.
\newblock {\em Annali di Matematica Pura ed Applicata (1923-)},
  204(4):1443--1468, 2025.

\bibitem{clapp2026entire}
M{\'o}nica Clapp and V{\'\i}ctor~A Vicente-Ben{\'\i}tez.
\newblock Entire solutions to a quasilinear purely critical competitive system.
\newblock {\em Journal of Differential Equations}, 450:113717, 2026.

\bibitem{Clapp_2024}
Mónica Clapp, Jorge Faya, and Alberto Saldaña.
\newblock Optimal pinwheel partitions for the {Y}amabe equation.
\newblock {\em Nonlinearity}, 37(10):105004, 2024.

\bibitem{clappRios}
Mónica Clapp and Luis {Lopez Rios}.
\newblock Entire nodal solutions to the pure critical exponent problem for the
  {$p$}-{L}aplacian.
\newblock {\em Journal of Differential Equations}, 265(3):891--905, 2018.

\bibitem{damascellimerchanmontorosciunzi}
Lucio Damascelli, Susana Merchán, Luigi Montoro, and Berardino Sciunzi.
\newblock Radial symmetry and applications for a problem involving the
  {$\Delta_p(\cdot)$} operator and critical nonlinearity in {$\RR^N$}.
\newblock {\em Advances in Mathematics}, 265:313--335, 2014.

\bibitem{musso}
Manuel {del Pino}, Monica Musso, Frank Pacard, and Angela Pistoia.
\newblock Large energy entire solutions for the {Y}amabe equation.
\newblock {\em Journal of Differential Equations}, 251(9):2568--2597, 2011.

\bibitem{del2013torus}
Manuel Del~Pino, Monica Musso, Frank Pacard, and Angela Pistoia.
\newblock Torus action on sn and sign-changing solutions for conformally
  invariant equations.
\newblock {\em Annali della Scuola Normale Superiore di Pisa-Classe di
  Scienze}, 12(1):209--237, 2013.

\bibitem{DolbeaultEsteban}
Jean Dolbeault and Maria~J. Esteban.
\newblock Rigidity versus symmetry breaking via nonlinear flows on cylinders
  and euclidean spaces.
\newblock {\em Inventiones mathematicae}, 206:397--440, 2016.

\bibitem{fernandezPalmasPetean}
Juan~Carlos {Fern\'andez}, Oscar Palmas, and Jimmy Petean.
\newblock Supercritical elliptic problems on the round sphere and nodal
  solutions to the {Yamabe} problem in projective spaces.
\newblock {\em Discrete and Continuous Dynamical Systems}, 40(4):2495--2514.

\bibitem{fernandez2020low}
Juan~Carlos Fern{\'a}ndez and Jimmy Petean.
\newblock Low energy nodal solutions to the {Y}amabe equation.
\newblock {\em Journal of Differential Equations}, 268(11):6576--6597, 2020.

\bibitem{ghoussoub-yuan}
N.~Ghoussoub and C.~Yuan.
\newblock Multiple solutions for quasi-linear pdes involving the critical
  {S}obolev and {H}ardy exponents.
\newblock {\em Transactions of the American Mathematical Society},
  352(12):5703--5743, 2000.

\bibitem{hewittRoss}
E.~Hewitt and K.A. Ross.
\newblock {\em Abstract Harmonic Analysis: Volume I Structure of Topological
  Groups Integration Theory Group Representations}.
\newblock Grundlehren der mathematischen Wissenschaften. Springer New York,
  1994.

\bibitem{LeLe}
Phuong Le and Diem Hang~T. Le.
\newblock Classification of positive solutions to p-{L}aplace equations with
  critical hardy–sobolev exponent.
\newblock {\em Nonlinear Analysis: Real World Applications}, 74, 2023.

\bibitem{medinaMusso}
Maria Medina and Monica Musso.
\newblock Doubling nodal solutions to the {Y}amabe equation in rn with maximal
  rank.
\newblock {\em Journal de Mathématiques Pures et Appliqu{\'e}es},
  152:145--188, 2021.

\bibitem{mercuri-willem}
Carlo Mercuri and Michel Willem.
\newblock A global compactness result for the p-laplacian involving critical
  nonlinearities.
\newblock {\em Discrete and continuous dynamical systems}, 28(2):469--493,
  2010.

\bibitem{ou2025classification}
Qianzhong Ou.
\newblock On the classification of entire solutions to the critical p-{L}aplace
  equation.
\newblock {\em Mathematische Annalen}, 392(2):1711--1729, 2025.

\bibitem{premo}
Bruno Premoselli and J{\'e}r{\^o}me V{\'e}tois.
\newblock Nonexistence of minimizers for the second conformal eigenvalue near
  the round sphere in low dimensions.
\newblock {\em arXiv preprint}, arXiv:2408.07823, 2024.

\bibitem{ramanujan2015collected}
Srinivasa Ramanujan.
\newblock Collected papers of srinivasa ramanujan.
\newblock 2015.

\bibitem{SCIUNZI201612}
Berardino Sciunzi.
\newblock Classification of positive {$\D^{1,p}(\RR^N)$}-solutions to the
  critical {$p$}-{L}aplace equation in {$\mathbb{R}^N$}.
\newblock {\em Advances in Mathematics}, 291:12--23, 2016.

\bibitem{sun2025criticalquasilinearequationsriemannian}
Linlin Sun and Youde Wang.
\newblock Critical quasilinear equations on riemannian manifolds, 2025.
\newblock arXiv:2502.08495.

\bibitem{szulkin2012sign}
Andrzej Szulkin and Shoyeb Waliullah.
\newblock Sign-changing and symmetry-breaking solutions to singular problems.
\newblock {\em Complex Variables and Elliptic Equations}, 57(11):1191--1208,
  2012.

\bibitem{vetois2016}
J\'er\^ome V\'etois.
\newblock A priori estimates and application to the symmetry of solutions for
  critical {$p$}-{L}aplace equations.
\newblock {\em Journal of Differential Equations}, 260(1):149--161, 2016.

\bibitem{vetois2024note}
J{\'e}r{\^o}me V{\'e}tois.
\newblock A note on the classification of positive solutions to the critical
  $p$-{L}aplace equation in $\mathbb{R}^n$.
\newblock {\em Advanced Nonlinear Studies}, 24(3):543--552, 2024.

\bibitem{ding}
Ding Weiyue.
\newblock On a conformally invariant elliptic equation on ${R}^n$.
\newblock {\em Commun. Math. Phys}, 107:331--335, 1986.

\bibitem{willem2012minimax}
Michel Willem.
\newblock {\em Minimax theorems}, volume~24.
\newblock Springer Science \& Business Media, 2012.

\end{thebibliography}
	
\end{document}